\documentclass[12pt,epsfig,amsfonts]{amsart}
\usepackage{amsmath,amsthm,amssymb,amscd,epsfig,color}
\usepackage{ulem}
\usepackage[all]{xy}
\usepackage{graphicx}
\usepackage{mathrsfs}
\usepackage{ulem}
\usepackage{pb-diagram} 

\newtheorem*{theorema}{Theorem A}
\newtheorem*{theoremb}{Theorem B}
\newtheorem*{theoremc}{Theorem C}

\newtheorem{prop}{Proposition}[section]
\newtheorem{lemma}[prop]{Lemma}

\newtheorem{thm}[prop]{Theorem}

\theoremstyle{definition}

\theoremstyle{remark}
\newtheorem{remark}[prop]{Remark}

\numberwithin{equation}{section}

\begin{document}

\author{Hiroki Takahasi}

\address{Keio Institute of Pure and Applied Sciences (KiPAS), Department of Mathematics,
Keio University, Yokohama,
223-8522, JAPAN} 
\email{hiroki@math.keio.ac.jp}



\title[Heterochaos baker maps and the Dyck system]
 {Exponential mixing for\\
  heterochaos baker maps and the Dyck system}
 \maketitle

 \begin{abstract}
 We investigate mixing properties of piecewise affine non-Markovian
 maps acting on $[0,1]^2$ or $[0,1]^3$ and preserving the Lebesgue measure, which are natural generalizations of the
 {\it heterochaos baker maps} introduced in [Y. Saiki, H. Takahasi, J. A. Yorke. Nonlinearity {\bf 34} (2021) 5744--5761].
 These maps are skew products over uniformly expanding or hyperbolic bases, and
 the fiber direction is a center in which both contracting and expanding behaviors coexist. 
 We prove that these maps are mixing of all orders. 
 For maps with a mostly expanding or contracting center, we establish exponential mixing 
 for H\"older continuous functions. 
 Using this result,
 for the Dyck system originating in 
 the theory of formal languages,
we establish exponential mixing 
 with respect to its two coexisting ergodic measures of maximal entropy.
    \end{abstract}

\section{introduction}
   The baker map (see \textsc{Figure}~\ref{bakermap})
\[(x,y)\in[0,1)^2\mapsto\begin{cases}\vspace{1mm}
    \displaystyle{\left(2x,\frac{y}{2}\right)}&\text{ on }\displaystyle{\left[0,\frac{1}{2}\right)\times[0,1)},\\
   \displaystyle{\left (2x-1,\frac{y+1}{2}\right)}&\text{ on }\displaystyle{\left[\frac{1}{2},1\right)\times[0,1)},
   \end{cases}\]
is 
one of the simplest uniformly hyperbolic dynamical systems.
The name `baker' is used since the action of the map is reminiscent of the kneading dough \cite{H56}.
The baker map preserves the Lebesgue measure on $[0,1)^2$, and mixing properties with respect to this measure are well-known: 
It is 
$K$
and hence mixing of all orders \cite{R63}; 
has exponential decay of correlations  
for H\"older continuous functions \cite{Bow75,Rue04}.

Is there any analogue of the baker map that provides a hands-on, elementary understanding of complicated phenomena in non-hyperbolic systems?
In response to this question,
 two piecewise affine maps on $[0,1]^2$ or $[0,1]^3$
 were introduced in \cite{STY21}, called heterochaos baker maps.
  The name `heterochaos'  comes from their distinctive property that periodic points with different unstable dimensions coexist densely 
  in an invariant set
  \cite[Theorem~1.1]{STY21}, a phenomenon aka {\it the unstable dimension variability} 
\cite{DGSY94,KKGOY}, which was recognized earlier \cite{AS70,BD96,Ma78,Si72} as a $C^1$ robust phenomenon for diffeomorphisms of closed manifolds of dimension at least three.

\begin{figure}
\begin{center}
\includegraphics[height=3.4cm,width=8.7cm]
{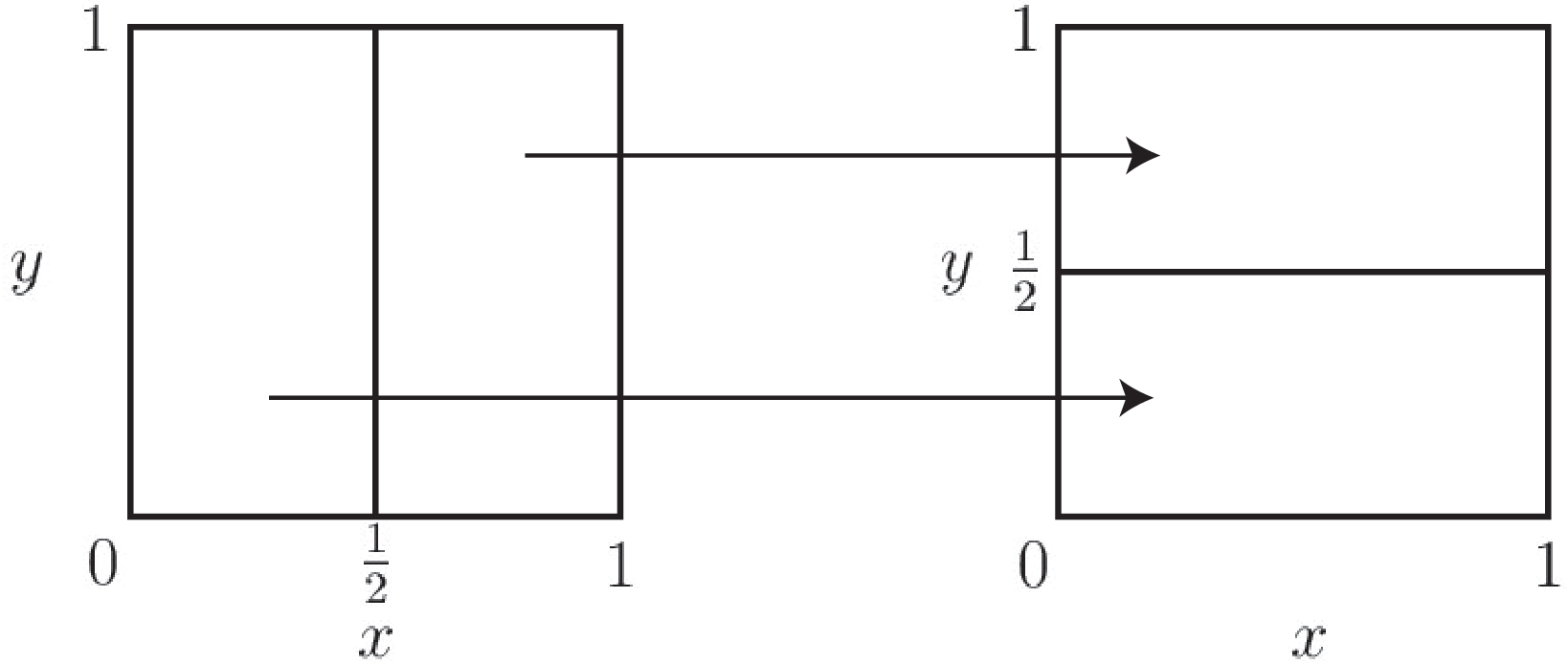}
\caption{The baker map: The $x$-direction is expanding by factor $2$ and the $y$-direction is contracting by factor $1/2$.}\label{bakermap}
\end{center}
\end{figure}

The analysis of the heterochaos baker maps is also motivated by the theory of symbolic dynamics,
which is closely related to automata theory and formal language theory.
There is a universal formal language due to W. Dyck. The Dyck system \cite{AU68,Kri74} is
the symbolic dynamics generated by that language. 
  Interestingly, the heterochaos baker maps are natural geometric models of the Dyck system \cite
  [Theorem~1.1]{TY}. This connection has opened the door for a complementary analysis of the heterochaos baker maps and the Dyck system, which consists in

   \begin{center}{\it verifying certain dynamical properties of the one side, with the aid of 
  the other.}
  \end{center}
  For a prototypical result on the heterochaos baker maps with the aid of the Dyck system,
  see \cite[Theorem~1.2]{TY}.

In this paper we proceed to
a complementary analysis of mixing properties of the heterochaos baker maps and the Dyck system. 
We will prove exponential mixing with respect to the Lebesgue measure for the heterochaos baker maps, and with the aid of this, prove exponential mixing for the Dyck system with respect to the two coexisting ergodic measures of maximal entropy.
Below we introduce these two dynamical systems, and state main results.

\subsection{The heterochaos baker maps}\label{hetero-b-sec} 
Let $m\geq2$ be an integer.
For $a\in(0,\frac{1}{m})$
 define $\tau_a\colon[0,1]\to[0,1]$ by 
\begin{equation}\label{Fa-1d}\tau_a(x)=\begin{cases}\vspace{1mm}
\displaystyle{\frac{x-(i-1)a}{a}}&\text{ on }[(i-1)a,ia),\ i\in\{1,\ldots,m\},\\ \displaystyle{\frac{x-ma}{1-ma}}&\text{ on }
[ma,1].
\end{cases}\end{equation}
 Consider the set of $2m$ symbols
 \begin{equation}\label{D}D=\{\alpha_1,\ldots,\alpha_m\}\cup\{\beta_1,\ldots,\beta_m\},\end{equation}
and define
\[\Omega_{\alpha_i}^+=\left[(i-1)a,ia\right)\times
\left[0,1\right]\text{ for }i\in\{1,\ldots,m\},\]
and
\[\Omega_{\beta_i}^+=\begin{cases}
\vspace{1mm}\displaystyle{\left[ma,1\right]\times
\left[\frac{i-1}{m },\frac{i }{m }\right)}&
\text{ for }i\in\{1,\ldots,m-1\},\\
\displaystyle{\left[ma,1\right]\times
\left[\frac{i-1 }{m},1\right]}&\text{ for }i=m.
\end{cases}\]
The sets $\Omega_{\alpha_i}^+$, 
$\Omega_{\beta_i}^+$, $i\in\{1,\ldots,m\}$ are pairwise disjoint and their union equals $[0,1]^2$.
Define a map $f_{a}\colon [0,1]^2\to[0,1]^2$ by
\begin{equation}\label{fa-2d}\begin{split}
  f_a(x,y)=
  \begin{cases}
    \displaystyle{\left(\tau_a(x),\frac{y}{m}+\frac{i-1}{m}\right)}&\text{ on }\Omega_{\alpha_i}^+,\ i\in\{1,\ldots,m\},\\
   \displaystyle{\left (\tau_a(x),my-i+m+1\right)}&\text{ on }\Omega_{\beta_i}^+,\ i\in\{1,\ldots,m\}.
   \end{cases}
\end{split}\end{equation}
See \textsc{Figure}~\ref{fig} for $m=2$.

\begin{figure}
\begin{center}
\includegraphics[height=3.4cm,width=9.5cm]
{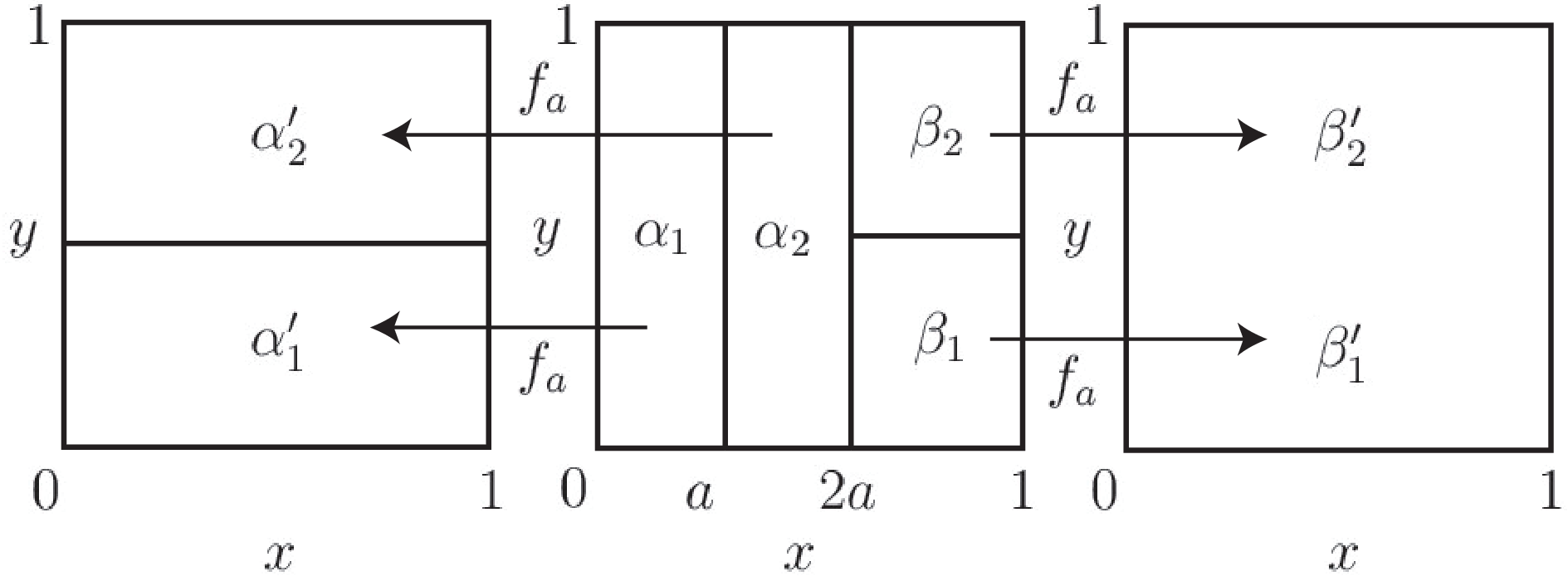}
\caption{The map $f_a$
for $m=2$: The domain  $\Omega_{\alpha_i}^+$  and 
 its image are labeled with $\alpha_i$ and 
 $\alpha_i'$ respectively, the same for $\Omega_{\beta_i}^+$; $f_a(\Omega_{\beta_1}^+)=[0,1]\times[0,1)$ and $f_a(\Omega_{\beta_2}^+)=[0,1]^2$.}\label{fig}
\end{center}
\end{figure}
\begin{figure}
\begin{center}
\includegraphics[height=4cm,width=9.5cm]
{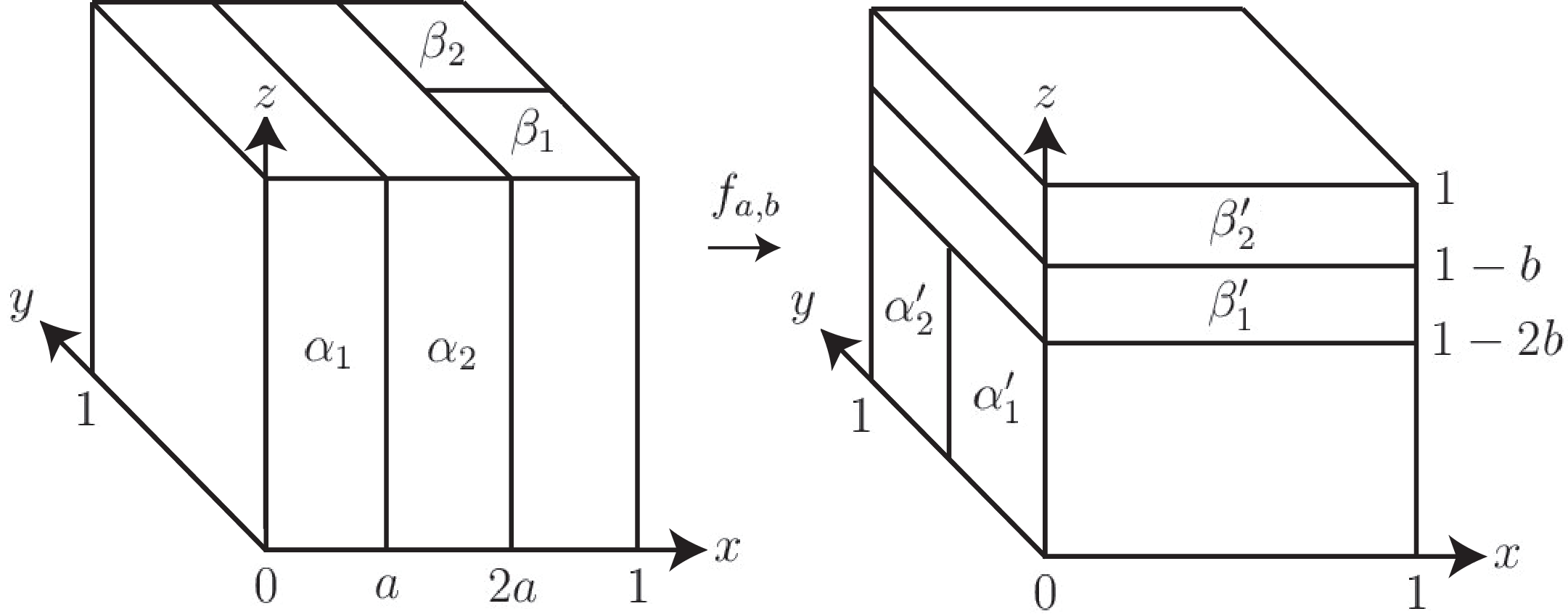}
\caption
{The map $f_{a,b}$ for $m=2$: The domain  $\Omega_{\alpha_i}^+$  and 
 its image are labeled with $\alpha_i$ and 
 $\alpha_i'$ respectively, the same for $\Omega_{\beta_i}^+$.}\label{fig2}
\end{center}
\end{figure}

Next, put $\Omega_{\alpha_i}=\Omega_{\alpha_i}^+\times\left[0,1\right]$ and $\Omega_{\beta_i}=\Omega_{\beta_i}^+\times\left[0,1\right]$
 for $i\in\{1,\ldots,m\}$. For $a,b\in(0,\frac{1}{m})$ 
 define a map $f_{a,b}\colon [0,1]^3\to[0,1]^3$ by 
\begin{equation}\label{3d-map}\begin{split}
  f_{a,b}(x,y,z)=
  \begin{cases}
    \displaystyle{\left(f_{a}(x,y),
    (1-mb)z\right)}&\text{ on }\Omega_{\alpha_i},\ i\in\{1,\ldots,m\},\\
   \displaystyle{\left (f_{a}(x,y),bz+1+b(i-m-1)\right)}&\text{ on }\Omega_{\beta_i},\ i\in\{1,\ldots,m\}.
   \end{cases}
\end{split}\end{equation}
 See \textsc{Figure}~\ref{fig2}.  
      In other words,
  $f_a$ is the projection of $ f_{a,b}$ to the $xy$-plane.
Note that $f_{a,b}$ is one-to-one except
on the points in the boundaries of $\Omega_{\alpha_i}$, $\Omega_{\beta_i}$ where it may be discontinuous and at most three-to-one. 
We call $f_a$ and $f_{a,b}$ 
{\it heterochaos baker maps}\footnote{The baker map is defined on $[0,1)^2$, whereas $f_a$ is defined on $[0,1]^2$. 
This minor discrepancy
has stemmed from the intention of the authors of \cite{STY21} to visualize the `heteroclinic set' in \cite{STY21}.}.

The heterochaos baker maps introduced in \cite{STY21} were precisely   $f_{\frac{1}{3}}$ and $f_{\frac{1}{3},\frac{1}{6}}$ with $m=2$. 
Symbolic dynamics for $f_a$, $f_{a,b}$ and their measures of maximal entropy were investigated in \cite{TY}. 
Saiki et al. \cite{STY22} considered some variants of $f_{a,b}$ in the context of fractal geometry and homoclinic bifurcations of diffeomorphisms.  Homburg and Kalle \cite{HK22} considered iterated function systems intimately related to $f_{a,b}$. 

For any $a\in(0,\frac{1}{m})$, $f_a$ preserves the Lebesgue measure on $[0,1]^2$. The map
 $f_{a,b}$ preserves the Lebesgue measure on $[0,1]^3$ if and only if $a+b=\frac{1}{m}$. Put
\[g_a=f_{a,\frac{1}{m}-a}.\]
 We will drop $a$ from notation and write $f$, $g$ and so on when the context is clear.

   \subsection{The Dyck system}\label{Dyck}
Krieger \cite{Kri00} introduced a certain class of subshifts having some algebraic property, called property A subshifts. It is a fundamental shift space in this class. 
The Dyck system is a subshift on the alphabet 
$D$ in \eqref{D}
consisting of $m$ brackets, $\alpha_i$ left and $\beta_i$ right in pair,
whose admissible words are sequences of legally aligned brackets.
To be more precise,
let $D^*$ denote the set of finite words in $D$.
Consider the monoid with zero, with $2m$ generators in $D$ and the unit element
$1$ with relations 
\begin{equation}\label{d-rel}\begin{split}\alpha_i\cdot\beta_j&=\delta_{i,j},\
0\cdot 0=0\text{ and }\\
\gamma\cdot 1&= 1\cdot\gamma=\gamma,\
 \gamma\cdot 0=0\cdot\gamma=0
\text{ for }\gamma\in D^*\cup\{ 1\},\end{split}\end{equation}
where $\delta_{i,j}$ denotes Kronecker's delta.
For $n\geq1$ and $\gamma_1\gamma_2\cdots\gamma_n\in D^*$ 
let
\[{\rm red}(\gamma_1\cdots\gamma_n)=\prod_{i=1}^n\gamma_i.\]
The one and two-sided Dyck shifts on $2m$ symbols are respectively defined by
\[\begin{split}\Sigma_{D}^+&=\{\omega\in D^{\mathbb Z_+}\colon {\rm red}(\omega_i\cdots \omega_j)\neq0\text{ for }i,j\in\mathbb Z_+\text{ with }i<j\},\\
\Sigma_{D}&=\{\omega\in D^{\mathbb Z}\colon {\rm red}(\omega_i\cdots \omega_j)\neq0\text{ for }i,j\in\mathbb Z\text{ with }i<j\},\end{split}\]
where $\mathbb Z_+$ denotes the set of non-negative integers.
Let $\sigma_+$ and $\sigma$ denote the left shifts acting on $\Sigma_{D}^+$ and $\Sigma_{D}$ respectively.

As a counterexample to the conjecture of Weiss \cite{W70},   Krieger \cite{Kri74} showed that the Dyck system has exactly two ergodic measures of maximal entropy, and
they are fully supported and Bernoulli.
Meyerovitch \cite{Mey08} proved the existence of tail invariant measures for the Dyck system.
To our knowledge, there is no result on
further statistical properties of the measures of maximal entropy for the Dyck system.

\subsection{Statements of the results}
 Under the iteration of  $f_a$,
the $x$-direction is expanding by factor $\frac{1}{a}$ or $\frac{1}{1-ma}$, while 
the $y$-direction is a center:
contracting by factor $\frac{1}{m}$ on $\Omega^+_\alpha$ and expanding by factor $m$ on $\Omega^+_\beta$,
where
\[\Omega^+_\alpha=\bigcup_{i=1}^m \Omega_{\alpha_i}^+\ \text{ and }\ \Omega^+_\beta=\bigcup_{i=1}^m\Omega_{\beta_i}^+.\]
The local stability in the $y$-direction along each orbit  
is determined by the asymptotic time average of 
the function
  \begin{equation}\label{geometric-c}\phi^c(p)=\begin{cases}-\log m&\text{ on }  \Omega^+_\alpha,\\
 \log m&\text{ on }  \Omega^+_\beta,\end{cases}\end{equation}
called {\it the central Jacobian}. Write $S_n\phi^c$ for the sum $\sum_{k=0}^{n-1}\phi^c\circ f_a^k$. 
Since $f_a$ is a skew product map over $\tau_a$
and $\phi^c$ is constant on sets of the form $\{x\}\times[0,1]$, $x\in[0,1]$, 
the ergodicity of $\tau_a$ with respect to the Lebesgue measure on $[0,1]$ implies that
  \begin{equation}\label{chia}\lim_{n\to\infty}\frac{1}{n}S_n\phi^c(p)=(1-2ma)\log m\ \text{ for Lebesgue a.e. $p\in[0,1]^2$.}\end{equation}
We classify $f_a$, $g_{a}$ into three types
according to the sign of this limit value:
\begin{itemize}
\item[(i)] $a\in(0,\frac{1}{2m})$ (mostly expanding center),

 \item[(ii)] $a\in(\frac{1}{2m},\frac{1}{m})$
 (mostly contracting center),
 \item[(iii)] $a=\frac{1}{2m}$ (mostly neutral center).
 \end{itemize}

In \cite{STY21},
 $f_{\frac{1}{3}}$ and $g_{\frac{1}{3}}$ with $m=2$ were shown to be weak mixing with respect to the Lebesgue measure.
  The proof there can be slightly modified 
  to show the weak mixingness of $f_a$ and $g_a$ with respect to the Lebesgue measure for any $a\in(0,\frac{1}{m})\setminus\{ \frac{1}{2m}\}$.
 Our first result considerably strengthens this.
Let $(X,\mathscr{B},\nu)$ be a probability space, $T\colon X\to X$ be a measurable map preserving $\nu$ and let $k\geq2$ be an integer.
We say $(T,\nu)$ is $k$-mixing if 
for all $A_0,A_1\ldots,A_{k-1}\in\mathscr{B}$,
\[\lim_{n_1,\ldots,n_{k-1}\to\infty}\nu(A_0\cap T^{-n_1}(A_1)\cap\cdots\cap T^{-n_1-n_2-\cdots-n_{k-1}}(A_{k-1}))=\prod_{j=0}^{k-1}\nu(A_j).\]
$2$-mixing is usually called mixing or strong mixing. 
Clearly $(k+1)$-mixing implies $k$-mixing, but the converse is unknown.
We say $(T,\nu)$ is {\it mixing of all orders} if it is $k$-mixing for any $k\geq2$. 
Let ${\rm Leb}$ denote the Lebesgue measure on $[0,1]^2$ or $[0,1]^3$.

\begin{theorema}
For any $a\in(0,\frac{1}{m})$, 
 $(f_a,{\rm Leb})$ and  $(g_a,{\rm Leb})$ are mixing of all orders.
\end{theorema}
In fact, we will show that 
$(f_a,{\rm Leb})$ is exact and the restriction of $(g_a,{\rm Leb})$
to a set of full Lebesgue measure has $K$-property (Proposition~\ref{exact}).
These properties imply mixing of all orders \cite[Section~2.6]{R63}.

Concerning rates of $2$-mixing
 we obtain 
 the following result. 
All functions appearing in this paper are real-valued.
For a pair $\varphi$, $\psi$ of functions in $L^2(\nu)$,
define their correlations by
\[{\rm Cor}_n(T;\varphi,\psi;\nu)=\left|\int\varphi(\psi\circ T^n){\rm d}\nu-\int\varphi{\rm d}\nu\int\psi{\rm d}\nu\right|\text{ for }n\geq1.\]
For a metric space $X$ and $\eta\in(0,1]$ let
$\mathscr H_\eta(X)$ denote the set of H\"older continuous functions on $X$ with a H\"older exponent $\eta$.
We say $(T,\nu)$ is {\it exponentially mixing}
if 
for each $\eta\in(0,1]$
there exists $\lambda=\lambda(\eta)\in(0,1)$ such that 
for any pair $\varphi$, $\psi$ of functions in $\mathscr H_\eta(X)$
there exists $C=C(\varphi,\psi)>0$ such that 
${\rm Cor}_n(T;\varphi,\psi;\nu)\leq C\lambda^n$ for all $n\geq1$.
\begin{theoremb}
For any $a\in(0,\frac{1}{m})\setminus\{\frac{1}{2m}\}$,
$(f_a,{\rm Leb})$ and $(g_a,{\rm Leb})$ are exponentially mixing.
\end{theoremb}

As we recall in Section~\ref{code-sec}, the heterochaos baker maps are geometric models of the Dyck system \cite[Theorem~1.1]{TY}: Following the orbits of $f_a$ and $f_{a,b}$ over the partitions $\{\Omega_\gamma^+\}_{\gamma\in D}$ and $\{\Omega_\gamma\}_{\gamma\in D}$, one obtains $\Sigma_D^+$ and $\Sigma_D$
respectively.
We exploit this connection
to establish exponential mixing for the Dyck system.

For the one-sided (resp. two-sided) Dyck shift,
there exist exactly two shift invariant ergodic measures of maximal entropy \cite{Kri74}, which we denote by $\nu_\alpha^+$ and $\nu_\beta^+$ (resp. $\nu_\alpha$ and $\nu_\beta$), leaving the details to Section~\ref{Dyck-mme}.
As usual, the metrics $d$ on the shift spaces are the Hamming metrics: For distinct points $\omega$, $\omega'\in\Sigma_D^+$, 
\[d(\omega,\omega')=\exp(-\min\{i\geq0\colon \omega_i\neq\omega_i'\}),\]
and for distinct points $\omega$, $\omega'\in\Sigma_D$, 
\[d(\omega,\omega')=\exp(-\min\{i\geq0\colon \omega_i\neq\omega_i'\text{ or }\omega_{-i}\neq\omega_{-i}'\}).\]
 \begin{theoremc}
 All
 $(\sigma_+,\nu_\alpha^+)$, $(\sigma_+,\nu_\beta^+)$, 
 $(\sigma,\nu_\alpha)$, $(\sigma,\nu_\beta)$ are exponentially mixing. 
\end{theoremc}

\subsection{Outline of proofs and the structure of the paper}\label{structure-sec} 
Uniformly hyperbolic systems have Markov partitions,
and their
statistical properties 
are well-understood on a symbolic level
\cite{Bow75,Rue04,Sin72}.  
 The difficulty
 in analyzing statistical properties of the heterochaos baker maps consists in the lack of a Markov partition,
which is precisely due to the dynamics in the $y$-direction.

To overcome this difficulty,
we combine inducing and large deviations.
A main tool in our construction is a Markov diagram, originally
introduced by Hofbauer \cite{Hof81,Hof86} and developed by Keller \cite{Ke89}, Buzzi \cite{Buz99} and so on.
In Section~2,
for $f_a$, $a\in(0,\frac{1}{m})$ we define a stopping time
$R\colon[0,1]^2\to\mathbb Z_+\cup\{\infty\}$, and
construct  a uniformly expanding induced Markov map $f^{R}$ with infinitely many branches.
We show that  
$\{R=\infty\}$ is a null set
for $a\in(0,\frac{1}{2m}]$ (Proposition~\ref{R-fin}),
and 
the Lebesgue measure of the tail 
$\{R>n\}$
decays exponentially in $n$ for $a\in(0,\frac{1}{2m})$ (Proposition~\ref{tail-eq0}).
Since the heterochaos baker maps are piecewise affine, all these construction and estimates are purely combinatorial
and distortion estimates are not needed.
In Section~3
we show that both $(f_{a},{\rm Leb})$ and $(g_{a},{\rm Leb})$ are mixing of all orders for all $a\in(0,\frac{1}{2m}]$.
Using the `invariance' of
 correlations under the replacement of $a$ by $\frac{1}{m}-a$
 (Proposition~\ref{dual}), we complete the proof of Theorem~A.
In Section~4 we construct {\it towers}, 
and apply Young's result \cite{You98} to deduce the 
exponential mixing for the heterochaos baker maps. In order to estimate error bounds, 
we perform a large deviation argument for the map $\tau_a$.
 Using the invariance of
 correlations again completes the proof of Theorem~B. 

A proof of Theorem~C in Section~5 relies on
a surprising connection between one of the two ergodic measures of maximal entropy for the Dyck system and the Lebesgue measure (Proposition~\ref{correspond}).
We emphasize that Theorem~C is not an immediate consequence of Theoerem~B, since 
the coding maps (see Section~\ref{code-sec} for the definition)
from the heterochaos baker maps into the Dyck system
do not preserve H\"older continuous functions.
Our idea is to transfer the towers 
for the heterochaos baker maps to the Dyck shift spaces, and apply Young's result \cite{You98} again
to obtain the exponential mixing with respect to one of the two ergodic measures of maximal entropy. Exponential mixing with respect to the other one
follows from the symmetry in the Dyck system (Proposition~\ref{dual-Dyck}).




\subsection{Comparison with results on partially hyperbolic systems}
Like the heterochaos baker maps, skew product maps over uniformly hyperbolic or expanding bases are considered to be simple models of partially hyperbolic 
systems\footnote{The splitting of the tangent bundle of $[0,1]^2$ into $x$- and $y$-directions does not always give rise to a partially hyperbolic splitting for $f_a$:
If $0<a<\frac{m-1}{m^2}$ then
the minimum expansion rate in the $x$-direction is
$\min\{\frac{1}{a},\frac{1}{1-ma}\}$, which does not dominate the expansion rate $m$ in the $y$-direction.}.
Hence, it is relevant to compare our results with others on exponential decay of correlations for partially hyperbolic systems. 

Young's method \cite{You98,You99}
of deducing nice statistical properties 
using towers with fast decaying tails has  been successfully applied to 
some partially hyperbolic systems, e.g.,
\cite{ALP05,C02,C04,Do00} to deduce exponential decay of correlations.
With more functional analytic methods,
exponential decay of correlations was proved for 
certain partially hyperbolic diffeomorphisms \cite{CN17}, and
 for certain piecewise partially hyperbolic endomorphisms \cite{BBL20}.  
Clearly the heterochaos baker maps are covered by none of these existing results.

\subsection{Comparison with results on subshifts}
For a large class of subshifts with `non-uniform specification', Climenhaga \cite{C18} constructed towers
and applied Young's method to obtain nice statistical properties including exponential decay of correlations.
Clearly the Dyck system does not satisfy the
`non-uniform specification', since it would imply the uniqueness of measure of maximal entropy. One can also directly check the breakdown of 'non-uniform specification' using the definition of the Dyck system.

We believe it would be difficult to construct 
towers for the Dyck system to which one can apply Young's result, purely within the symbolic setting
with no recourse to the connection with the heterochaos baker maps
proved in \cite{TY}.
The geometric approach to the Dyck system in the proof of Theorem~C based on this connection seems promising, in order to establish many other dynamical properties of the measures of maximal entropy for the Dyck system. Moreover, this approach may be extended to some other subshifts with a high symmetry,
such as the Motzkin system \cite{M04}.

\section{Dynamics of the heterochaos baker maps}
In this section we analyze the dynamics of the heterochaos baker map $f_a$, $a\in(0,\frac{1}{m})$.  
 In Section~\ref{Hofbauer-sec} we introduce a Markov diagram for $f_a$ and describe its combinatorial structure. In Section~\ref{count-sec} we develop path counting arguments in the diagram, 
 and in Section~\ref{interpret} provide useful formulas for the Birkhoff sums of the function $\phi^c$ in \eqref{geometric-c}. 
  In Section~\ref{a-p} we introduce a stopping time $R$,
   and express it as `a twice jump time under a constraint'.
   In Section~\ref{another-sec} we give another formula for the stopping time, and
 show in Section~\ref{fin-title} the almost sure finiteness of the stopping time for $a\in(0,\frac{1}{2m}]$.
In Section~\ref{ind-exp-sec} we construct an induced Markov map,
and in Section~\ref{stop-sec} estimate the Lebesgue measure of the tail of the stopping time.
In Section~\ref{code-sec}
we recall the connection between the heterochaos baker maps and the Dyck system established in \cite{TY}.

\subsection{A Markov diagram}\label{Hofbauer-sec} 
Let ${\rm int}$ denote the interior operation in $\mathbb R^2$ or $\mathbb R^3$.
Let $A\subset[0,1]^2$ be a non-empty set with $A\subset {\rm int}(\Omega_\gamma^+)$ for some $\gamma\in D$. A set
  $B\subset[0,1]^2$ is a {\it successor} of $A$ if $B=f(A)\cap {\rm int}(\Omega_{\gamma'}^+)$ holds for some $\gamma'\in D$.
  If $B$ is a successor of $A$,  we write $A\to B$.
 We set
  \[\mathcal V_\alpha=\{{\rm int}(\Omega_{\alpha_i}^+)\colon i\in\{1,\ldots,m\} \}\ \text{ and }\ \mathcal V_\beta=\{{\rm int}(\Omega_{\beta_i}^+)\colon i\in\{1,\ldots,m\} \},\]
  and define $\mathcal V_n$, $n=0,1,\ldots$ inductively by
  \[\mathcal V_n=\begin{cases}\mathcal V_\alpha\cup\mathcal V_\beta&\text{ for }n=0,\\
  \mathcal V_{n-1}\cup\{B\colon \text{$B$ is a successor of an element of $\mathcal V_{n-1}$}\}&\text{ for }n\geq1.\end{cases}\]
  We set $\mathcal V=\bigcup_{n=0}^{\infty}\mathcal V_n$.
  The oriented graph $(\mathcal V,\to)$ is called a {\it Markov diagram} for $f$.
  We decompose the set $\mathcal V$ of vertices into infinitely many levels $\mathcal L_r$,
  $r=-1,0,1,\ldots$ as follows:
 \[\mathcal L_r=\begin{cases}\mathcal V_\alpha&\text{ for }r=-1,\\
\mathcal V_1\setminus \mathcal V_\alpha&\text{ for }r=0,\\
 \mathcal V_{r+1}\setminus\mathcal V_r&\text{ for }r\geq1.\end{cases}\]
    See \textsc{Figure}~\ref{diagram}. 
    Define 
 $l\colon \mathcal V\to\mathbb Z$ by 
 \[v\in \mathcal L_{l(v)}\text{ for each }v\in\mathcal V.\]
    

\begin{figure}
\begin{center}
\includegraphics[height=7cm,width=13.5cm]
{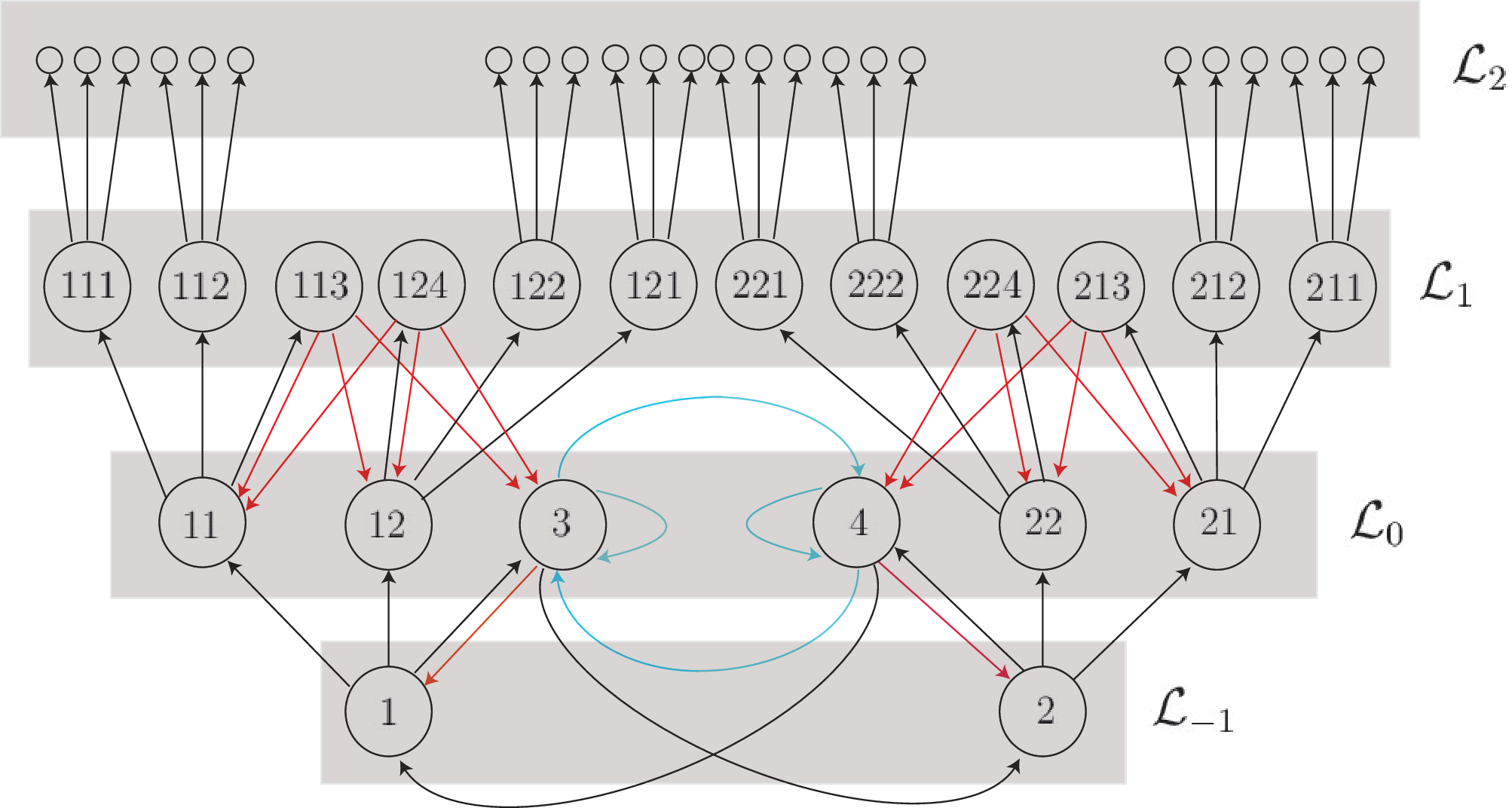}
\caption{Part of the Markov diagram $(\mathcal V,\to)$ for $m=2$.
We put $\alpha_1=1$, $\alpha_2=2$, $\beta_1=3$, $\beta_1=4$ for simplicity.
The vertices ${\rm int}(\Omega_{\gamma_1}^+)$,  $f({\rm int}(\Omega_{\gamma_1}^+))\cap{\rm int}(\Omega_{\gamma_2}^+)$,
$f(f({\rm int}(\Omega_{\gamma_1}^+))\cap {\rm int}(\Omega_{\gamma_2}^+))\cap {\rm int}(\Omega_{\gamma_3}^+)$ 
are labeled with $\gamma_1$, $\gamma_1\gamma_2$, $\gamma_1\gamma_2\gamma_3$ respectively.}
\label{diagram}
\end{center}
\end{figure}


It is possible to describe all outgoing edges and incoming edges for each vertex in $\mathcal V$.
However, we only have to estimate from above 
the number of paths of a given length connecting two given vertices.
For this purpose,
it suffices to describe the cardinalities of all outgoing edges from each vertex $v\in\mathcal V$ which are either upward, downward or parallel (shown in black, red, blue arrows respectively in \textsc{Figure}~\ref{diagram}), in terms of the function $l$.
The cardinalities of the three sets
\[\begin{split}
\mathcal V_{\uparrow}(v)&=\{v'\in\mathcal V \colon v\to v',\ l(v')= l(v)+1\},\\
\mathcal V_{\downarrow}(v)&=\{v'\in\mathcal V \colon v\to v',\ l(v')= l(v)-1\},\\
\mathcal V_{\to}(v)&=\{v'\in\mathcal V \colon v\to v',\  l(v')= l(v)\},
\end{split}\]
are given as follows:

\begin{itemize}
\item[(i)] For each  
$v\in \mathcal V_\alpha$,
\[\#\mathcal V_{\uparrow}(v)=m+1,\ \#\mathcal V_{\downarrow}(v)=0,\ \#\mathcal V_{\to}(v)=0.\]

\item[(ii)] For each  $v\in \mathcal V_\beta$,
\[\#\mathcal V_{\uparrow}(v)=0,\
 \#\mathcal V_{\downarrow}(v)=m,\  
 \#\mathcal V_{\to }(v)=m.\]

\item[(iii)]  For each $v\in\mathcal V\setminus\mathcal V_0$ with $v\subset \Omega^+_\alpha$,
\[\#\{v'\in\mathcal V_{\uparrow}(v)\}=m+1,\  \#\mathcal V_{\downarrow}(v)=0,\  
 \#\mathcal V_{\to}(v)=0.\]
\item[(iv)] For each $v\in\mathcal V\setminus\mathcal V_0$ with $v\subset \Omega^+_\beta$,
\[\#\{v'\in\mathcal V_{\uparrow}(v)\}=0,\  \#\mathcal V_{\downarrow}(v)=m+1,\  
 \#\mathcal V_{\to}(v)=0.\]


\end{itemize}

\subsection{Counting paths in the Markov diagram}\label{count-sec}
Let $n\geq1$.
A {\it path} of length $n$ is
a word $v_0\cdots v_{n}$ of elements of $\mathcal V$ of 
word length $n+1$ satisfying
 $v_k\to v_{k+1}$ for every $k\in\{0,\ldots,n-1\}$.
We say $k\in\{0,\ldots,n-1\}$ is a {\it hold time}
of a path $v_0\cdots v_{n}$ if $\{v_{k},v_{k+1}\}\subset\mathcal V_\beta$. 
Let $P_n$ denote the set of paths
 of length $n$ which have no hold time.
 
    In order to estimate the cardinalities of $P_n$ and its various subsets,
we consider projections to paths of the symmetric random walk on $\mathbb Z$.
Let 
\begin{equation}\label{zn-def}Z_n=\left\{l_0\cdots
l_n\in\mathbb Z^{n+1}\colon 
|l_{k}-l_{k+1}|=1 \text{ for }k\in\{0,\ldots,n-1\}\right\}.\end{equation}
Define a projection $\Phi_n\colon P_n\to Z_n$ by 
\[\Phi_n(v_0v_1\cdots v_n) = l(v_0)l(v_1)\cdots l(v_n).\]

\begin{lemma}\label{count}
Let $n\geq1$ and
  let $l_0l_1\cdots l_n\in Z_n$. 
    We have
\[
\#\Phi_n^{-1}(l_0l_1\cdots l_n)\leq (m+1)m^{\frac{1}{2}(n+l_n-l_0)}.\]
\end{lemma}
\begin{proof}
   A path $u_0\cdots u_j$ in the Markov diagram is {\it upward}
if $l(u_k)<l(u_{k+1})$ for all $k\in\{0,\ldots,j-1\}$, and {\it downward}
if $l(u_{k+1})>l(u_k)$ for all $k\in\{0,\ldots,j-1\}$.
Let $r$, $s$ be integers with $-1\leq r<s$. 
From the description of the Markov diagram
in Section~\ref{Hofbauer-sec},
the number of upward paths from one vertex in $\mathcal L_r$ to 
another in $\mathcal L_s$ does not exceed 
$m^{s-r-1}(m+1)$, and
 the number of upward paths from one vertex in $\mathcal L_r$ to 
another in $\mathcal L_s$ which can be concatenated to a downward path does not exceed 
$m^{s-r-1}$.
Reciprocally, 
the number of downward paths from one vertex in $\mathcal L_s$ to another in $\mathcal L_r$ 
does not exceed 
$m+1$, and
the number of downward paths from one vertex in $\mathcal L_s$ to another in $\mathcal L_r$ which can be concatenated to an upward path 
does not exceed 
$m$.

Let $l_0\cdots l_n\in Z_n$. 
We may assume $l_k\geq-1$
for $k\in\{0,\ldots,n\}$ for otherwise the desired inequality is trivial.
Any path in $\Phi_n^{-1}(l_0\cdots l_n)$
is uniquely written as
the alternate concatenation of upward and downward paths, see \textsc{Figure}~\ref{ascend} for example.
The sum of the lengths of all the upward paths in this concatenation equals $\frac{1}{2}(n+l_n-l_0)$. Hence we obtain
the desired inequality. 
   \end{proof}

 \begin{figure}
\begin{center}
\includegraphics[height=6.5cm,width=13cm]
{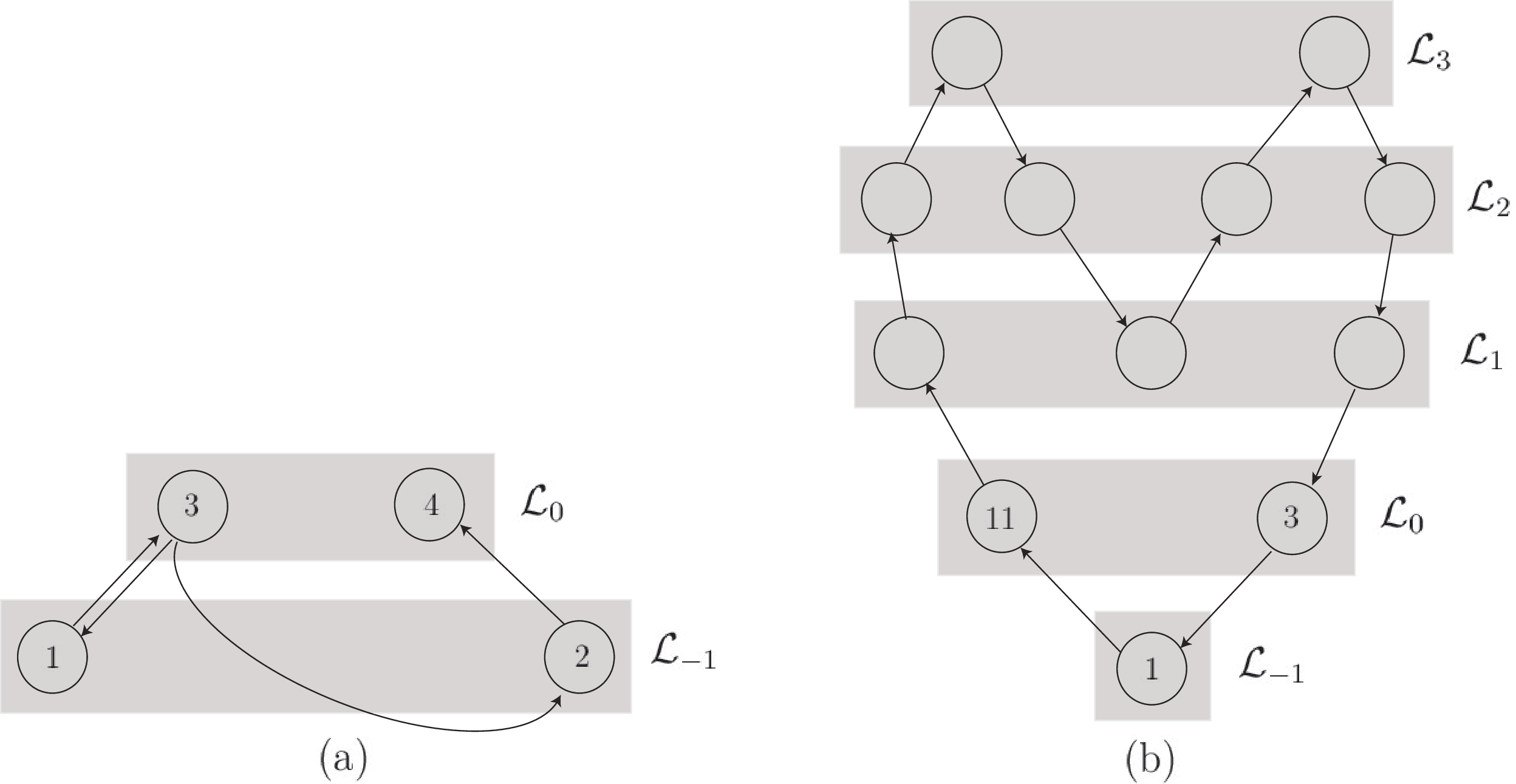}
\caption{Alternate concatenations of upward and downward paths
for $m=2$ with $\alpha_1=1$, $\alpha_2=2$, $\beta_1=3$, $\beta_1=4$ for simplicity:\\
(a) $3\to1\to2\to4$; (b) $1\to11\to\cdots\to3\to1$.}\label{ascend}
\end{center}
\end{figure}

For a vertex $v\in \mathcal L_0\setminus\mathcal V_\beta$ and positive integers $n$, $j$,
let $P_n(j;v)$ denote the set of paths
    $v_0\cdots v_{n}$ of length $n$ such that $v_0=v$,
    $l(v_k)\geq1$ for all $k\in\{1,\ldots,n\}$ and $l(v_n)=j$.
    Note that $P_n(j;v)\neq\emptyset$ if and only if 
    $n\equiv j$ mod $2$. 
 
 \begin{lemma}\label{1st-passage}
    For any $v\in\mathcal L_0\setminus\mathcal V_\beta$ and all positive integers $n$, $j$ with $n\equiv j$ mod $2$, 
    we have
     \[\#P_n(j;v)\leq\frac{j}{n}\begin{pmatrix}n\\\frac{n+j}{2}\end{pmatrix}\frac{m+1}{m}m^{\frac{n+j}{2}}.\]
   \end{lemma}
   \begin{proof}

   Since each path in $P_n(j;v)$ has no hold time,
   each element of $\Phi_n(P_n(j;v))$ 
   may be viewed a path of the symmetric random walk on $\mathbb Z$ starting at the position $j$
   which hits the origin $0$ at step $n$ for the first time.
The reflection principle for the random walk gives
   \begin{equation}\label{hit2}\#\Phi_n(P_n(j;v))=
   \frac{j}{n}\begin{pmatrix}n\\\frac{n+j}{2}\end{pmatrix}.\end{equation}
By Lemma~\ref{count}, for each $v_0\cdots v_n\in P_n(j;v)$ we have
   \begin{equation}\label{hit3}\#\Phi_n^{-1}(l(v_0)\cdots l(v_n))\leq \frac{m+1}{m}m^{\frac{n+j}{2}}.\end{equation}
   Combining \eqref{hit2} and \eqref{hit3} we obtain the desired inequality.\end{proof}


\subsection{Formulas for Birkhoff sums of the central Jacobian}\label{interpret}
Let $n\geq1$ and let $v_0\cdots v_n$ be a path in the Markov diagram $(\mathcal V,\to)$.
The Birkhoff sum $S_n\phi^c$ on the rectangle $\bigcap_{k=0}^{n}f^{-k}(v_k)$ is a constant,
which we denote by $S_n\phi^c(v_0\cdots v_n)$ with a slight abuse of notation.
From the structure of the Markov diagram described in Section~\ref{Hofbauer-sec}, we derive two useful formulas for this constant.

\begin{itemize}
\item[(i)] If $v_0\cdots v_n\in P_n$, then 
\begin{equation}\label{interpret-eq1}S_n\phi^c(v_0\cdots v_n)=(l(v_0)-l(v_n))\log m.\end{equation}
  \item[(ii)] If $v_0\in\mathcal V_\beta$ and $v_n\in\mathcal V_\beta$ then
\begin{equation}\label{interpret-eq2}S_n\phi^c(v_0\cdots v_n)=\#\{k\in\{0,\ldots,n-1\}\colon\text{$k$ is a hold time of $v_0\cdots v_n$}  \}\log m.\end{equation}
\end{itemize}

Since any path in $(\mathcal V,\to)$ can be written as a concatenation of paths of the above two kinds, the Birkhoff sum of the central Jacobian along any path is kept in track by a combination of the formulas \eqref{interpret-eq1} and \eqref{interpret-eq2}.

\subsection{Stopping time}\label{a-p}
We denote by $|\cdot|$ the Lebesgue measure on $[0,1]$, $[0,1]^2$, $[0,1]^3$. A {\it rectangle} (resp. {\it block}) 
is a product of two (resp. three) 
non-degenerate intervals in $[0,1]$.
For a rectangle $B=B_x\times B_y$ we write \[
|B|_x=|B_x|\ \text{ and }\ |B|_y=|B_y|.\]
Similarly, for a block $B=B_x\times B_y\times B_z$ we write
 \[|B|_x=|B_x|,\ |B|_y=|B_y|\ \text{ and }\ |B|_z=|B_z|.\]

For $p\in [0,1]^2$ and $n\geq1$, 
let $K_n(p)$ denote the maximal rectangle containing $p$ on which $f^n$ is affine.
Clearly 
 $|K_1(p)|_y=1$ for $p\in\Omega^+_\alpha$,
 $|K_1(p)|_y=\frac{1}{m}$ for $p\in\Omega^+_\beta$
  and
 $|K_{n}(p)|_y/|K_{n-1}(p)|_y\in\left\{1,\frac{1}{m}\right\}$ for $n\geq2$.
An integer $n\geq2$ is a {\it cutting time} of $p$
if $|K_n(p)|_y/|K_{n-1}(p)|_y=\frac{1}{m}$.
We define a {\it stopping time}
  \begin{equation}\label{afnew-eq} R(p)=\inf\{n\geq2\colon \text{$n$ is a cutting time of $p$}\}.\end{equation}
  Clearly we have $R(p)=2$ if and only if $p$ is contained in the set
 \[E=\Omega^+_\beta\cap f^{-1}(\Omega^+_\beta).\]

The stopping time is expressed as follows.
Put $S_{0}\phi^c\equiv 0$ for convenience. 
\begin{prop}
\label{character}
For all $p\in[0,1]^2$ we have 
\[R(p)=\begin{cases}
\inf\{n\geq 0
\colon S_n\phi^c(p)=-\log m\text{ and }f^n(p)\in E \}+2
&\text{ on }\Omega^+_\alpha,\\
\inf\{n\geq 0
\colon S_n\phi^c(p)=0\text{ and }f^n(p)\in E \}+2
&\text{ on }\Omega^+_\beta.
\end{cases}\]
\end{prop}
Before proceeding let us record two equalities. First,
\begin{equation}\label{afnew-eq20}
f(\Omega^+_\beta\setminus E)=\Omega^+_\alpha.
\end{equation}
The definition of the stopping time
 \eqref{afnew-eq} implies 
 \begin{equation}\label{afnew-eq10}R\circ f=R-1\ \text{
on }\Omega^+_\beta\setminus E.\end{equation}

\begin{proof}[Proof of Proposition~\ref{character}]
 By 
\eqref{afnew-eq20} and \eqref{afnew-eq10}  
it suffices to consider the case $p\in \Omega^+_\beta\setminus E$. 
Clearly we have $R(p)\geq3$.
We claim $S_{R(p)-2}\phi^c(p)=0$, for otherwise \eqref{interpret-eq2} implies
 $S_{R(p)-2}\phi^c(p)>0$, and 
there is $k\in\{2,\ldots,R(p)-1\}$ such that $f^{k-2}(p)\in E$ and $S_{k-2}\phi^c(p)=0$. From Lemma~\ref{chara-lem} below, it follows that $k$ is a cutting time of $p$, a contradiction to the minimality in the definition of $R(p)$. This claim and \eqref{interpret-eq1} together imply $S_{n}\phi^c(p)<0$ for all $n\geq1$ with $n\leq R(p)-3$, and so the desired equality holds.
\end{proof}
\begin{lemma}\label{chara-lem}
Let $p\in\Omega^+_\beta$ and $n\geq2$.
Then $n$ is a cutting time of $p$
if and only if 
$f^{n-2}(p)\in E$
and $S_{n-2}\phi^c(p)\geq0$.
\end{lemma}

\begin{proof}
To show the `if' part, suppose
    $n\geq2$, $f^{n-2}(p)\in E$
and $S_{n-2}\phi^c(p)\geq0$.
Take $i\in\{1,\ldots,m\}$ with
$f^{n-2}(p)\in\Omega_{\beta_i}^+$.
Let $B$ denote the connected component of $f^{-n+2}(\Omega_{\beta_i}^+)$ that contains $p$.
Then $B$ is a rectangle, $f^{n-2}|_B$ is affine and $|f^{n-2}(B)|_y=\frac{1}{m}$. Hence  $n$ is a cutting time of $p$.
To show the `only if' part, let $n\geq2$ be a cutting time of $p$. The definition of cutting time implies
      $f^{n-2}(p)\in E$. 
    Since $p\in\Omega^+_\beta$,
     by \eqref{interpret-eq2} we obtain $S_{n-2}\phi^c(p)\geq0$.
\end{proof}

\subsection{Formula for the stopping time in terms of pullbacks}\label{another-sec}
Let $A$ be a non-empty subset of $[0,1]^2$ and let $n\geq1$. A connected component of $f^{-n}(A)$
is called a {\it pullback} of $A$ by $f^n$. If $B$ is a pullback of $A$ by $f^n$ and $f^n|_B$ is affine, we call $B$ an affine pullback of $A$ by $f^n$,
or simply an {\it affine pullback}.
If $A$ is connected and $B$ is an affine pullback of $A$ by $f^n$, then $f^n(B)=A$.

 \begin{prop}
 \label{exist-cor}
For any $p\in(0,1)^2$ such that 
$R(p)$ is finite and $f^{R(p)}(p)\in(0,1)^2$, we have
\[R(p)=\min\left\{\begin{split}&n\geq2\colon\text{there exists an affine pullback $B$ of $(0,1)^2$ by $f^n$}\\
&\quad\quad\quad\ \text{such that }
p\in B\text{ and }|B|_y=\frac{1}{m}|K_1(p)|_y\end{split}\right\}.\]
\end{prop}


\begin{proof}
In the case
$R(p)=2$, we have $p\in E$ and the desired equality is obvious. Suppose $R(p)\geq3$. Then $p\notin E$ holds.
By Proposition~\ref{character}, there exists
 $i\in\{1,\ldots,m\}$ such that
$f^{ R(p)-2}(p)\in\Omega^+_{\beta_i}\cap E$. 
    The pullback of the rectangle
$\Omega_{\beta_i}^+\cap E$
by $f^{R(p)-2}$
that contains $p$, denoted by $B'$, is an affine pullback and satisfies
 $|B'|_y=|K_1(p)|_y$.
  Proposition~\ref{character} implies
 ${\rm int}(\Omega_{\beta_i}^+\cap E)\subset f^{R(p)-2}(B')$.
 In particular,
 $f^{R(p)-1}|_{B'}$ is affine
and ${\rm int}(\Omega^+_\beta)\subset f^{R(p)-1}(B')$.
From the assumption $f^{R(p)}(p)\in(0,1)^2$, there exists an affine pullback $B$ of $(0,1)^2$ by $f^{R(p)}$ satisfying $p\in B\subset B'$ and $|B|_y=\frac{1}{m}|B'|_y=\frac{1}{m}|K_1(p)|_y$. Hence,
$R(p)$ does not exceed the minimum 
 in the right-hand side of the desired equality in the proposition. The reverse inequality is obvious from the property of $B$.
\end{proof}

\subsection{Almost sure finiteness of the stopping time}\label{fin-title}
Clearly, if $a\in(\frac{1}{2m},\frac{1}{m})$ then the set $\{R=\infty\}$ has positive Lebesgue measure. Otherwise this is a null set.

\begin{prop}\label{R-fin}
If $a\in(0,\frac{1}{2m}]$ then the stopping time 
 is finite 
 Lebesgue a.e. on $[0,1]^2$.
\end{prop}
\begin{proof}
Proposition~\ref{character} and \eqref{interpret-eq1}, \eqref{interpret-eq2}
together imply 
$\limsup_{n\to\infty}\frac{1}{n}S_n\phi^c(p)\leq0$
for all $p\in\{R=\infty\}$.
 Recall \eqref{chia} that $\frac{1}{n}S_n\phi^c$
converges almost surely to the positive constant $(1-2ma)\log m$.
Hence, if $a\in(0,\frac{1}{2m})$ then $R$ is finite Lebesgue a.e. on $[0,1]^2$.

For the rest of the proof of Proposition~\ref{R-fin} we assume $a=\frac{1}{2m}$. We need two lemmas.

 \begin{lemma}\label{CLT}
For Lebesgue a.e. $x\in[0,1]$, there exists a pair $(M,N)$
 of positive integers 
 such that \[\sup_{n\geq N}\frac{1}{\sqrt{n}}|S_n\phi^c(x,0)|\leq M.\]
 \end{lemma}
 \begin{proof}Put
$v(\phi^c)=\int_{[0,1]}|\phi^c(x,0)|^2{\rm d}x$. Let $(a_k)_{k=2}^\infty$ be the increasing sequence of positive reals satisfying
$(1/\sqrt{2\pi v(\phi^c) })\int_{\mathbb R\setminus[-a_k,a_k]} \exp\left(-\frac{x^2}{2v(\phi^c)}\right){\rm d}x=1/k^2$ for all $k\geq2$.
For $k\geq2$ and $n\geq1$, define
\[G_{k,n}=\left\{x\in[0,1]\colon \frac{1}{\sqrt{n}}S_n\phi^c(x,0)\in [-a_k,a_k]\right\}.\]
The central limit theorem applied to the sequence of independently identically distributed random variables
$x\in[0,1]\mapsto \phi^c(f^j(x,0))$ $(j=0,1,\ldots)$ with mean $0$ and variance $v(\phi^c)$ yields
     \[\lim_{n\to\infty}|G_{k,n}|=\frac{1}{\sqrt{2\pi v(\phi^c) }}\int_{[-a_k,a_k]} \exp\left(-\frac{x^2}{2v(\phi^c) }\right){\rm d}x=1-\frac{1}{k^2}.\]
Choose a sequence $(n_k)_{k=2}^\infty$ of positive integers such that $\sum_{k=2}^\infty|[0,1]\setminus\bigcup_{n=n_k}^\infty G_{k,n}|<\infty$.
 By Borel-Cantelli's lemma,      \begin{equation}\label{BC}\left|\limsup_{k\to\infty}\left([0,1]\setminus \bigcup_{n=n_k}^\infty G_{k,n}\right)\right|=0.\end{equation}

    Let $G$ denote the set of $x\in[0,1]$ for which there is no positive integer pair $(M,N)$ satisfying $\sup_{n\geq N}|S_n\phi^c(x,0)|/\sqrt{n}\leq M$.
    For 
     any $x\in G$ we have $|S_n\phi^c(x,0)|/\sqrt{n}\to\infty$ as $n\to\infty$.
     Hence, for any $k\geq2$ there is $n\geq n_k$ such that
$|S_n\phi^c(x,0)|/\sqrt{n}> a_k$, and so
$x\in [0,1]\setminus \bigcup_{n=n_k}^\infty G_{k,n}$.
Since $k\geq2$ is 
arbitrary, \eqref{BC} yields $|G|=0$, which verifies 
Lemma~\ref{CLT}.
 \end{proof}

 Define $Q\colon [0,1]^2\setminus E\to\mathbb Z_+\cup\{\infty\}$ by
 \[Q(p)=\begin{cases}
\inf\{n\geq 1
\colon S_n\phi^c(p)=-\log m\text{ and }f^n(p)\in \Omega^+_\beta \}
&\text{ on }\Omega^+_\alpha,\\
\inf\{n\geq 1
\colon S_n\phi^c(p)=0\text{ and }f^n(p)\in \Omega^+_\beta \}
&\text{ on }\Omega^+_\beta\setminus E.
\end{cases}\]
 \begin{lemma}\label{Q-finite}
$Q$ is finite
Lebesgue a.e. on $[0,1]^2\setminus E$.\end{lemma}
\begin{proof}
By \eqref{afnew-eq20}
and $Q\circ f=Q-1$ on $\Omega^+_\beta\setminus E$,
 it suffices to show that $Q$ is finite
Lebesgue a.e. on
$\Omega^+_\alpha$.
Recall that $f_{\frac{1}{2m}}$ is a skew product over the piecewise affine fully branched map $\tau_{\frac{1}{2m}}$ in \eqref{Fa-1d} that preserves the Lebesgue measure on $[0,1]$,
and $\phi^c$ is constant on the sets $\{x\}\times[0,1]$, $x\in[0,1]$.
By Lemma~\ref{CLT},
 for Lebesgue a.e. $p\in[0,1]^2$ there exists a positive integer pair $(M,N)$ satisfying $\sup_{n\geq N}|S_n\phi^c(p)|/\sqrt{n}\leq M$.
 Moreover, the definition of $Q$ and the formulas \eqref{interpret-eq1}, \eqref{interpret-eq2}
together imply 
 $\sup_{k\geq1}S_k\phi^c(p)\leq-\log m$
 for all $p\in \Omega^+_\alpha\cap\{Q=\infty\}$.
 Hence, Lebesgue a.e. point in
$\Omega^+_\alpha\cap\{Q=\infty\}$ is contained in the set
 \begin{equation}\label{meas-inc}W=\bigcup_{M=1}^\infty\bigcup_{N=1}^\infty\bigcap_{n=N}^\infty W_{M,n},\end{equation}
 where $M$, $n$ are positive integers and \[W_{M,n}=\left\{p\in \Omega^+_\alpha\colon \sup_{k\geq 1}S_k\phi^c(p)\leq -\log m\
\text{ and }\ S_n\phi^c(p)\geq-M\sqrt{n}\right\}.\] 

 It suffices to show that $W$ is a null set.
For each $M\geq1$ we have
\[W_{M,n}\subset\bigcup_{v\in\mathcal L_0\setminus\mathcal V_\beta}\bigcup_{j=1}^{\lfloor M\sqrt{n}\rfloor} \bigcup_{v_0\cdots v_n\in P_n(j;v)}  \bigcap_{k=0}^{n}f^{-k}(v_k),\]
where $\lfloor\cdot\rfloor$ denotes the floor function.
Using Lemma~\ref{1st-passage} to bound $\#P_n(j;v)$
and the identity $|\bigcap_{k=0}^{n}f^{-k}(v_k)|_x=2^{-n}m^{-\frac{n+j}{2}}$ for each path
$v_0\cdots v_n$ in $P_n(j;v)$, we have
\[\begin{split}|W_{M,n}|&\leq\sum_{v\in\mathcal L_0\setminus\mathcal V_\beta}\sum_{v_0\cdots v_n\in P_n(j;v)}\left|\bigcap_{k=0}^{n}f^{-k}(v_k)\right|\\
&\leq\sum_{j=1}^{\lfloor M\sqrt{n}\rfloor}\frac{j}{n}\begin{pmatrix}n\\
\frac{n+j}{2}\end{pmatrix} \frac{m+1}{m}m^{\frac{n+j}{2}}2^{-n}m^{-\frac{n+j}{2}}
\leq n^{-\frac{3}{2}}\sum_{j=1}^{\lfloor M\sqrt{n}\rfloor}j\leq\frac{M^2}{\sqrt{n}},
\end{split}\]
for all sufficiently large $n$.
For the third inequality we have used Stirling's formula for factorials to evaluate the binomial coefficient.
This yields
$|\bigcap_{n=N}^\infty W_{M,n}|\leq\inf_{n\geq N}|W_{M,n}|=0$ for all $M$, $N\geq1$, and therefore
\[\begin{split}\left|W\right|&\leq\sum_{M=1}^\infty\sum_{N=1}^\infty
\left|\bigcap_{n=N}^\infty W_{M,n}\right|=0,\end{split}\]
as required. This completes the proof of Lemma~\ref{Q-finite}.
\end{proof}

By virtue of \eqref{afnew-eq20}, \eqref{afnew-eq10} and $R=2$ on $E$, it suffices to show
that $R$ is finite Lebesgue a.e. on 
$\Omega^+_\beta\setminus E$.
Let $2\mathbb Z_+$ denote the set of positive even integers. 
If $p\in\Omega^+_\beta\setminus E$ and
 $Q(p)$ is finite then $Q(p)\in 2\mathbb Z_+$.
  Put $M_1=E$, and
 $M_k=\Omega^+_\beta\cap\{Q=k\}$
 for $k\in 2\mathbb Z_+$.
 For each $k\in \{1\}\cup 2\mathbb Z_+$,
$M_k$ is written as a finite union $M_k=\bigcup_\ell M_{k,\ell}$ of pairwise disjoint rectangles
$M_{k,\ell}$ with the properties that
$M_{k,\ell}\subset\Omega^+_{\beta_i}$ for some $i\in\{1,\ldots,m\}$, $|M_{k,\ell}|_y=\frac{1}{m}$, 
$f^{k}|_{M_{k,\ell}}$ is affine and
 ${\rm int}(\Omega^+_{\beta_j})\subset f^{k}(M_{k,\ell})\subset\Omega^+_{\beta_j}$ for some $j\in\{1,\ldots,m\}$.
Lemma~\ref{Q-finite} gives
$|\Omega^+_\beta\setminus\bigcup_{k\in\{1\}\cup 2\mathbb Z_+} M_k|=0$.

The map $H\colon\bigcup_{k\in\{1\}\cup2\mathbb Z_+}M_k\to \Omega^+_\beta$ given by $H(x,y)=f^{k}(x,y)$ for $(x,y)\in M_k$
has the skew product form \[H(x,y)=(H_1(x),H_2(x,y)).\] The map $H_1$ is a piecewise affine map on the interval $[ma,1]$ with infinitely many full branches,
leaving the normalized Lebesgue measure on the interval invariant and ergodic.
From Birkhoff's ergodic theorem for $H_1$, 
$n_1(p)=\inf\{n\geq1\colon H^n(p)\in E\}$ is finite
for Lebesgue a.e. $p\in \Omega^+_\beta\setminus E$. Moreover,
 \eqref{interpret-eq1} implies $S_{n_1(p)}\phi^c(p)=0$
 for $p\in \Omega^+_\beta\setminus E$.
This together with Proposition~\ref{character} 
 implies $R(p)=n_1(p)+2<\infty$
 for Lebesgue a.e. $p\in\Omega^+_\beta\setminus E$. The proof of Proposition~\ref{R-fin} is complete.
\end{proof}

\subsection{An induced Markov map}\label{ind-exp-sec}
For each $n\geq1$
let $\mathscr{P}_{n}$ denote 
the collection of affine pullbacks 
of $(0,1)^2$ by $f^n$ which are contained in $\{R=n\}$.
We set \[\mathscr{P}=\bigcup_{n=1}^\infty\mathscr{P}_n.\]
Elements of $\mathscr{P}$ are pairwise disjoint, and the stopping time is constant on each element. 
We now define an induced map $f^{R}\colon \bigcup_{\omega\in\mathscr{P}}\omega\to[0,1]^2$ by $f^{R}|_\omega=f^{R|_\omega}|_\omega$ on each $\omega\in\mathscr{P}$, and put 
\[\Delta_0^+=\bigcap_{n=0}^\infty(f^{R})^{-n}\left(\bigcup_{\omega\in\mathscr{P} }\omega\right).\]  
 Partition $\Delta^+_0$ into $\{\omega\cap\Delta^+_0\}_{\omega\in\mathscr{P}}$,
 label the partition elements with an integer $i\geq1$, and
write
$\{\Delta^+_{0,i}\}_{i=1}^\infty=\{\omega\cap\Delta^+_0\}_{\omega\in\mathscr{P}}$
and $R_i=R|_{\Delta^+_{0,i}}.$
We have $\Delta^+_{0}=\bigcup_{i=1}^\infty\Delta^+_{0,i}$,
and $f^{R}$ maps each $\Delta_{0,i}^+$ affinely onto $\Delta_0^+$.
Proposition~\ref{R-fin} implies $|\Delta_0^+|=1$ for $a\in(0,\frac{1}{2m}]$.

\subsection{Stopping time estimates}\label{stop-sec}
For $a\in(0,\frac{1}{m})$ put
\[\chi(a)=-\log\sqrt{a(1-ma)}.\]
Note that $\chi(a)\geq\log\sqrt{4m}>0$.
Moreover we have $\sqrt{4m}e^{-\chi(a)}\leq1$, and the equality holds if and only if $a=\frac{1}{2m}$.
\begin{prop}\label{tail-eq0}There exists $n_0\geq1$ such that
for all $a\in(0,\frac{1}{2m}]$ and all $n\geq n_0$ we have
\[|\{R=n+2\}|\leq n^{-\frac{3}{2}}(\sqrt{4m}e^{-\chi(a)})^n.\]
In particular, if $a\in(0,\frac{1}{2m})$ then $|\{R=n+2\}|$ decays exponentially in $n$.
    \end{prop}

To prove Proposition~\ref{tail-eq0}, we estimate the size of each element of the partition
$\{\Delta^+_{0,i}\}_{i=1}^\infty$,
and the cardinality of the set of elements with a given stopping time.
\begin{lemma}\label{area}
For all $a\in(0,\frac{1}{2m}]$ and all $i\geq1$,
    \[
 |\Delta^+_{0,i}|\leq
      \exp(-\chi(a)(R_i-2)).\]
\end{lemma}
\begin{proof}
 For $i\geq1$ write 
$\Delta_{0,i}^{+}=\omega\cap\Delta_0^+$,
$\omega\in\mathscr P$ and put
\[\begin{split}R_{i,\alpha}&=\#\left\{0\leq k\leq R_i-1\colon f^k(\omega)\subset\Omega^+_\alpha\right\},\\
R_{i,\beta}&=\#\left\{0\leq k\leq R_i-1\colon f^k(\omega)\subset\Omega^+_\beta\right\}.\end{split}\]
Clearly we have $R_{i,\alpha}+R_{i,\beta}=R_i.$
 Proposition~\ref{exist-cor} implies the following:
 \begin{itemize}
\item[(i)] If $\omega\subset\Omega^+_\alpha$, then 
 $|\omega|_y=\frac{1}{m}$ and
  $S_{R_i}\phi^c=\log m$ on $\omega$.
  In particular, $R_i$ is odd and
   \[R_{i,\alpha}=\frac{1}{2}(R_i-1)\ \text{ and }\ R_{i,\beta}=\frac{1}{2}(R_i+1).\]

\item[(ii)] If $\omega\subset\Omega^+_\beta$, then 
  $|\omega|_y=\frac{1}{m^2}$ and
   $S_{R_i}\phi^c=2\log m$ 
   on $\omega$.
   In particular, $R_i$ is even and
      \[R_{i,\alpha}=\frac{1}{2}(R_i-2)\ \text{ and }\ R_{i,\beta}=\frac{1}{2}(R_i+2).\] 
      \end{itemize}
      See \textsc{Figure}~\ref{fig-image}.
We also have
 $\log |\omega|_x=R_{i,\alpha}\log a+R_{i,\beta}\log(1-ma).$
 Combining this with the above (i) (ii) yields
$|\omega|_x\leq
\exp(-\chi(a)(R_i-2)).$
 Since
$|\Delta^+_{0,i}|=|\omega|_x
|\omega|_y$, we obtain the desired inequality.
 \end{proof}


\begin{figure}
\begin{center}
\includegraphics[height=5cm,width=15cm]
{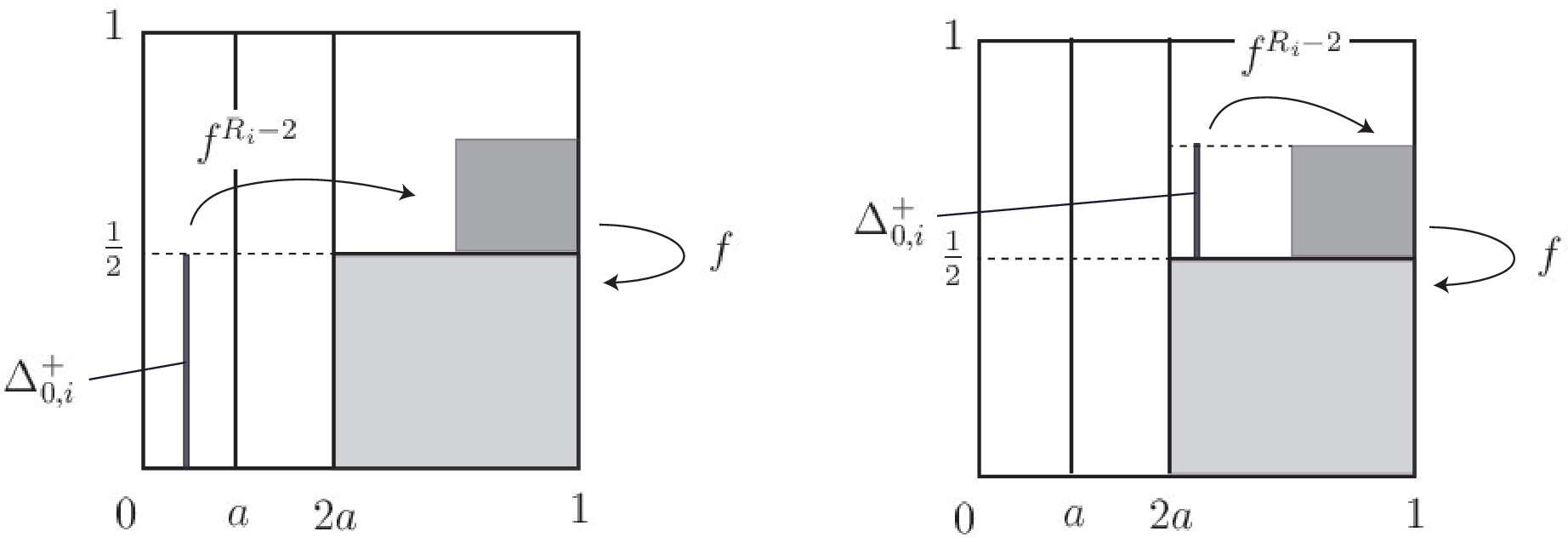}
\caption
{The images of
$\Delta_{0,i}^+$, $i\geq1$ for $m=2$: $\Delta_{0,i}^+\subset\Omega^+_\alpha$ (left); $\Delta_{0,i}^+\subset\Omega^+_\beta$ (right);
$|[0,1]^2\setminus f^{R_i}(\Delta_{0,i}^+)|=0$ in both cases.}\label{fig-image}
\end{center}
\end{figure}

 For each $n\geq1$, define
  \[P_n^*=\begin{cases}\{v_0\cdots v_{n}\in P_n\colon l(v_0)=-1\text{ and }
 l(v_{n})=0 \}&\text{ if $n$ is odd,}\\
 \{v_0\cdots v_{n}\in P_n\colon l(v_0)=0\text{ and }
 l(v_{n})=0 \}&\text{ if $n$ is even.}\end{cases}\]

 \begin{lemma}\label{path-number}
    For all $n\geq2$ we have      
        \[\# P_n^*\leq
       \frac{2(m+1)}{n+1}\begin{pmatrix}n+2\\\lfloor\frac{n+4}{2}\rfloor\end{pmatrix}m^{\lfloor\frac{n+1}{2}\rfloor}.\]
   \end{lemma}
 \begin{proof}  
For integers $n$, 
$s$, $t$ with $n\geq1$ and $s\leq t$,
let 
\[Z_n(s,t)=\{l_0\cdots l_n\in Z_n\colon l_0=s,\ l_n=t\}.\]
If $s< t$, then let
\[Z_n^*(s,t)=\{l_0\cdots l_n\in Z_n(s,t)\colon 
l_k>s \text{ for } k\in\{1,\ldots,n\}\}.\]
Note that
$Z_n^*(s,t)\neq\emptyset$ if and only if
 $t-s\equiv n$ mod $2$. 
The reflection principle of the symmetric random walk on $\mathbb Z$ 
gives
\begin{equation}\label{z0pas}\#Z_n^*(s,t)=\frac{t-s}{n}\begin{pmatrix}n\\\frac{n+
t-s}{2}\end{pmatrix}.\end{equation}

 By Lemma~\ref{count}, for each $v_0\cdots v_n\in P_n^*$ we have
\begin{equation}\label{z0paseq1}\#\Phi_n^{-1}(l(v_0)\cdots l(v_n))\leq\begin{cases} (m+1)m^{\frac{n+1}{2}}&\text{ if $n$ is odd,}\\
(m+1)m^{\frac{n}{2}}&\text{ if $n$ is even.}\end{cases}\end{equation}
Suppose $n\geq2$ is odd. For each $v_0\cdots v_n\in P_n^*$ we have $l(v_0)\cdots l(v_n)\in  Z_n(-1,0).$
Since the map
  $l_0\cdots l_{n}\in Z_n(-1,0)\mapsto  (-2)l_0\cdots l_{n}\in Z_{n+1}^*(-2,0)$ is injective, \eqref{z0pas} gives
\begin{equation}\label{z0paseq2}\#\Phi_n(P_n^*)\leq\#Z_{n+1}^*(-2,0)\leq\frac{2}{n+1}\begin{pmatrix}n+1\\\frac{n+3}{2}\end{pmatrix}.\end{equation}
Suppose
  $n\geq2$ is even. Similarly, for each $v_0\cdots v_n\in P_n^*$ we have $l(v_0)\cdots l(v_n)\in Z_n(0,0).$
Since the map
  $l_0\cdots l_{n}\in Z_n(0,0)\mapsto  (-2)(-1)l_0\cdots l_{n}\in Z_{n+2}^*(-2,0)$ is injective,  \eqref{z0pas} gives
 \begin{equation}\label{z0paseq3}\#\Phi_n(P_n^*)\leq
 \#Z_{n+2}^*(-2,0)\leq \frac{2}{n+2}\begin{pmatrix}n+2\\\frac{n+4}{2}\end{pmatrix}.\end{equation}
  Combining \eqref{z0paseq1}, \eqref{z0paseq2} and \eqref{z0paseq3}
 yields the desired inequality in the lemma.
   \end{proof}

\begin{proof}[Proof of Proposition~\ref{tail-eq0}]
Let $n\geq2$. For each $i\geq1$ 
    with $R_i=n+2$, there exists a unique
    path $v_0\cdots v_{n}\in P_n^*$ such that   
     $f^k(\Delta^+_{0,i})\subset v_k$ for $0\leq k\leq n$. 
By Lemmas~\ref{area} and \ref{path-number} we have
  \[
  \begin{split}|\{R=n+2\}|&=\sum_{i\geq1\colon R_i=n+2 }|\Delta^+_{0,i}|\leq\#P_n^*e^{-\chi(a)n}\\
  &\leq \frac{2(m+1)}{n+1}\begin{pmatrix}n+2\\
  \lfloor\frac{n+4}{2}\rfloor\end{pmatrix}m^{\lfloor\frac{n+1}{2}\rfloor}
 e^{-\chi(a)n}\leq n^{-\frac{3}{2}}(\sqrt{4m}e^{-\chi(a)})^n,\end{split}\]
provided $n$ is large enough.
    To deduce the last inequality we have evaluated the binomial coefficient using Stirling's formula for factorials.
  \end{proof}


  \subsection{Symbolic dynamics}\label{code-sec}
For $a,b\in(0,\frac{1}{m})$ we put
\[\Lambda_a=\bigcap_{n=0 }^{\infty} f_a^{-n}\left(\bigcup_{\gamma\in D}{\rm int}( \Omega_\gamma^+)\right)\quad\text{ and }\quad \Lambda_{a,b}=\bigcap_{n=-\infty }^{\infty}f_{a,b}^{-n}\left(\bigcup_{\gamma\in D}{\rm int}(\Omega_\gamma)\right).\]
 Define coding maps $\pi_{a}\colon (x,y)\in \Lambda_a\mapsto
(\omega_n)_{n\in\mathbb Z_+}\in D^{\mathbb Z_+}$ and $\pi_{a,b}\colon (x,y,z)\in \Lambda_{a,b}\mapsto
 (\omega_n)_{n\in\mathbb Z}\in D^{\mathbb Z}$ by \[(x,y)\in 
\bigcap_{n=0 }^\infty f_a^{-n}({\rm int}(
\Omega_{\omega_n}^+))\ \text{ and }\ 
 (x,y,z)\in\bigcap_{n=-\infty }^\infty f_{a,b}^{-n}({\rm int}
  (\Omega_{\omega_n})).\]
   Note that $\sigma_+\circ\pi_{a}=\pi_{a}\circ f_{a}$ and  $\sigma\circ\pi_{a,b}=\pi_{a,b}\circ f_{a,b}.$
The coding maps are not injective. The next connection between the heterochaos baker maps and the Dyck system lies at the basis of the proof of Theorem~C.

   \begin{thm}[\cite{TY}, Theorem~1.1]\label{TY-thm}
For all $a,b\in(0,\frac{1}{m})$
we have
\[\Sigma_{D}^+= \overline{\pi_{a}(\Lambda_a)}\ \text{ and }\ \Sigma_{D}=
   \overline{\pi_{a,b}(\Lambda_{a,b})}.\]
\end{thm}

\section{Mixing of all orders for the heterochaos baker maps}
In this section we prove Theorem~A.
In Section~\ref{3d} we show a certain invariance of correlations,
which implies the sufficiency to consider only $a\in(0,\frac{1}{2m}]$.
In Section~\ref{K} we establish the exactness and $K$-property of the corresponding maps, and in Section~\ref{proofthma} complete the proof of Theorem~A.

\subsection{Invariance of correlations}\label{3d}

Let $(X,\mathscr{B},\nu)$ be a probability space and let $T\colon X\to X$ be a measurable map preserving $\nu$.
For an integer $k\geq2$ and functions 
$\phi_0, \phi_1,\ldots,\phi_{k-1}\in L^k(\nu)$, consider their
correlations
 \[{\rm Cor}_{n_1,\ldots,n_{k-1}}(T;\phi_0,\ldots,\phi_{k-1};\nu)= \left|\int\prod_{j=0}^{k-1}(\phi_j\circ T^{n_j}) {\rm d}\nu-\prod_{j=0}^{k-1}\int\phi_j {\rm d}\nu\right|,\]
 where $0=n_0\leq n_1\leq\cdots \leq n_{k-1}$. 
  It is well-known that
 $(T,\nu)$ is $k$-mixing if and only if ${\rm Cor}_{n_1,\ldots,n_{k-1}}(T;\phi_0,\ldots,\phi_{k-1};\nu)\to0$ as $n_1,n_2-n_1,\ldots,n_{k-1}-n_{k-2}\to\infty$ for all $\phi_0, \phi_1,\ldots,\phi_{k-1}\in L^k(\nu)$.

 Define an involution $\iota\colon[0,1]^3\to[0,1]^3$ by
\[\iota(x,y,z)=(1-z,1-y,1-x).\] 

\begin{prop}\label{dual}
Let $a\in(0,\frac{1}{m})$.
For all $k\geq2$, $\phi_0,\ldots,\phi_{k-1}\in L^k({\rm Leb})$ and integers $n_0,\ldots,n_{k-1}$ with
$0=n_0\leq n_1\leq n_2\leq\cdots \leq n_{k-1}$
we have
\[\begin{split}{\rm Cor}_{n_1,\ldots,n_{k-1}}&(g_{a};\phi_0,\ldots,\phi_{k-1};{\rm Leb})\\
&={\rm Cor}_{n_{k-1},n_{k-1}-n_{1},n_{k-1}-n_2,\ldots,0}(g_{\frac{1}{m}-a};\phi_0\circ\iota^{-1},\ldots,\phi_{k-1}\circ\iota^{-1};{\rm Leb}).\end{split}\]
\end{prop}
\begin{proof}
Define $g_{a}^*\colon[0,1]^3\to[0,1]^3$ by
\begin{equation}\label{f*}g_{a}^*=\iota^{-1}\circ g_{\frac{1}{m}-a}\circ \iota.\end{equation}
 For Lebesgue a.e. $p\in[0,1]^3$ we have
  $g_a^*(g_a(p))=p$. Moreover,
  for Lebesgue a.e. $p\in[0,1]^3$ we have
\begin{equation}\label{3d-eq19}
\begin{split}\phi_j\circ g_{a}^{n_j}(p)&=\phi_j\circ (g_{a}^*)^{n_{k-1}-n_j}\circ g_a^{n_{k-1}-n_j}\circ g_a^{n_j}(p)\\
&=\phi_j\circ (g_a^*)^{n_{k-1}-n_j}\circ g_a^{n_{k-1}}(p)\ \text{ for every }j\in\{0,\ldots,k-1\}.\end{split}\end{equation} Using \eqref{3d-eq19} and the $g_a$-invariance of ${\rm Leb}$, \eqref{f*} and the $\iota$-invariance of ${\rm Leb}$ yield
\[\begin{split}\int\prod_{j=0}^{k-1}(\phi_j\circ g_a^{n_j}) {\rm d}{\rm Leb}&=\int\prod_{j=0}^{k-1}(\phi_j\circ (g_a^*)^{n_{k-1}-n_j}) {\rm d}{\rm Leb}\\&=\int\prod_{j=0}^{k-1}(\phi_j\circ \iota^{-1}\circ g_{\frac{1}{m}-a}^{n_{k-1}-n_j}\circ\iota) {\rm d}{\rm Leb}\\&=\int\prod_{j=0}^{k-1}(\phi_j\circ \iota^{-1}\circ g_{\frac{1}{m}-a}^{n_{k-1}-n_j}) {\rm d}{\rm Leb}.\end{split}\]
We also have 
$\int\phi_j{\rm d}{\rm Leb}=\int\phi_j\circ\iota^{-1}{\rm d}{\rm Leb}$, and 
so the desired equality holds.
\end{proof}

\subsection{Exactness and $K$-property}\label{K}
Let $(X,\mathscr{B},\nu)$ be a probability space.
For sub-sigma-algebras $\mathscr{C}$, $\mathscr{D}$ of $\mathscr{B}$
we write $\mathscr{C}\stackrel{\circ}{\subset}\mathscr{D}$ if
for every $C\in\mathscr{C}$ there exists $D\in\mathscr{D}$ such that $\nu(C\ominus D)=0$, where $\ominus$ denotes the symmetric difference of sets. We write $\mathscr{C}\stackrel{\circ}{=}\mathscr D$
if $\mathscr{C}\stackrel{\circ}{\subset}\mathscr{D}$ and $\mathscr{D}\stackrel{\circ}{\subset}\mathscr{C}$.
If $\{\mathscr{B}_n\}_{n=1}^\infty$ is a family of sub-sigma-algebras of $\mathscr{B}$,
let $\bigvee_{n=1}^\infty\mathscr{B}_n$ denote the smallest sub-sigma-algebra that contains all the $\mathscr{B}_n$.
If $\{\mathscr{B}_n\}_{n=1}^\infty$ is a family of partitions of $X$ into measurable sets,
let $\bigvee_{n=1}^\infty\mathscr{B}_n$ denote the smallest sub-sigma-algebra that contains all the $\mathscr{B}_n$.

Let $T\colon X\to X$ be a measurable map preserving $\nu$. For a sub-sigma-algebra $\mathscr{C}$ of $\mathscr{B}$ and $n\geq1$,
let $T^{-n}\mathscr{C}=\{T^{-n}(B)\colon B\in\mathscr{C}\}$. If $T$ has a measurable inverse, 
let $T^{n}\mathscr{C}=\{T^{n}(B)\colon B\in\mathscr{C}\}$.
 We say $(T,\nu)$ is {\it exact} if $T$ has no measurable inverse and $\bigcap_{n=0}^\infty T^{-n}\mathscr B\stackrel{\circ}{=}\{X,\emptyset\}$. 
We say $(T,\nu)$ has {\it $K$-property} if $T$ has a measurable inverse and there exists a sub-sigma-algebra $\mathscr{K}$ of $\mathscr{B}$ such that:
\begin{itemize}
\item[(i)] 
 $\mathscr{K}\subset T\mathscr{K}$.
\item[(ii)] 
 $\bigvee_{n=0}^{\infty}T^n\mathscr{K}\stackrel{\circ}{=}\mathscr B$.
\item[(iii)] 
 $\bigcap_{n=0}^\infty T^{-n}\mathscr{K}\stackrel{\circ}{=}\mathscr \{X,\emptyset\}$.
\end{itemize}
If $(T,\nu)$ has $K$-property, $T$ is usually called a Kolmogorov automorphism \cite{Wal82}.
Exactness or $K$-property implies mixing of all orders \cite[Section~2.6]{R63}.

\begin{prop}\label{exact}
For any $a\in(0,\frac{1}{2m}]$,
     $(f_a,{\rm Leb})$ is exact
and
     $(g_a|_{\Lambda_a,\frac{1}{m}-a},{\rm Leb}|_{\Lambda_{a,\frac{1}{m}-a}})$ has
 $K$-property. 
\end{prop}


For a proof of this proposition
we need some notation and one preliminary lemma.
For $a\in(0,\frac{1}{m})$ define $\mathscr{A}=\mathscr{A}_a$ by
\[\mathscr{A}=\bigcup_{n=1}^\infty\{B_0\cap (f^R)^{-1}(B_1)\cap\cdots\cap (f^R)^{-n+1}(B_{n-1})\colon B_0,\ldots,B_{n-1}\in\mathscr{P}|_{\Delta_0^+}\}.\]
Let $\mathscr{B}(X)$ denote the Borel sigma-algebra on a topological space $X$.
\begin{lemma}\label{density}
If $a\in(0,\frac{1}{2m}]$, then for any $A\in\mathscr{B}([0,1]^2)$ with positive Lebesgue measure  
and any $\varepsilon\in(0,1)$,
there exists $\omega\in\mathscr A$ 
such that $|A\cap \omega|\geq(1-\varepsilon)|\omega|$.
\end{lemma}
\begin{proof}

Since the collection of unions of countably many elements of $\mathscr A$ is an algebra
on $\Delta^+_0$, 
for any $\varepsilon>0$
there exist 
 finitely many elements $\omega_{1},\ldots,\omega_{k}$ of $\mathscr A$ such that 
$|\bigcup_{j=1}^k \omega_{j}\ominus A|<\varepsilon^2$.
If $\varepsilon$ is sufficiently small, 
there exists $j_0\in\{1,\ldots,k\}$ such that
 $|A\cap \omega_{j_0}|\geq(1-\varepsilon)|\omega_{j_0}|$, for otherwise we obtain 
 the following contradiction:
$|A|<(1-\varepsilon)\sum_{j=1}^k|\omega_j|+|A\setminus\bigcup_{j=1}^k \omega_j|<(1-\varepsilon)(|A|+\varepsilon^2)+\varepsilon^2<|A|$.
\end{proof}


\begin{proof}[Proof of Proposition~\ref{exact}]
Let $A\in \bigcap_{n=0}^\infty f^{-n}\mathscr B([0,1]^2)$
satisfy $|A|>0$. To verify the exactness of $(f,{\rm Leb})$ it suffices to show that $|A|=1$.
 By Lemma~\ref{density}, for any $\varepsilon>0$ there exist $n\geq1$ and an affine pullback of $\omega$ of $(0,1)^2$ by $f^n$
 such that $|A\cap \omega|>(1-\varepsilon)|\omega|$.
 There exists  $A'\in\mathscr B([0,1]^2)$ such that
 $A=f^{-n}(A')$, and therefore $A'\supset f^n(A)$ and
 \[|A|=|A'|\geq|f^n(A)|\geq|f^n(A\cap \omega)|\geq(1-\varepsilon)|f^n(\omega)|=1-\varepsilon.\]
 The last inequality is because $f^n|_\omega$ is affine.
 Since $\varepsilon$ is arbitrary we obtain $|A|=1$.

    Note that ${\rm Leb}(\Lambda)=1$, 
  $g|_\Lambda$ has a measurable inverse, and the restriction ${\rm Leb}|_{\Lambda}$ is $g|_{\Lambda}$-invariant.
Let $\mathscr{K}$ denote the smallest sub-sigma-algebra of $\mathscr{B}(\Lambda)$ that contains $\{(A\times[0,1])\cap \Lambda\colon A\in\mathscr{A}\}$. 
From the definition of the map \eqref{3d-map} we obtain
$\mathscr{K}\subset g\mathscr{K}$ and $\bigvee_{n=0}^{\infty}g^n\mathscr{K}\stackrel{\circ}{=}\mathscr B(\Lambda)$.
The argument in the previous paragraph shows
$\bigcap_{n=0}^\infty g^{-n}\mathscr{K}\stackrel{\circ}{=}\mathscr \{\Lambda,\emptyset\}$. We have verified that
 $(g|_{\Lambda },{\rm Leb}|_{\Lambda })$ has
 $K$-property. 
\end{proof}
\begin{remark} The exactness of $(f_a,{\rm Leb})$ for $a\in(0,\frac{1}{2m})$ can be shown 
by applying \cite[Theorem~1~(iii)]{You99} (see also \cite[Lemma~5]{You98})
to the tower map $F$
introduced in Section~\ref{tower-sec}.
This argument does not work 
 for $a=\frac{1}{2m}$ 
 since the lift of the Lebesgue measure to the tower 
 becomes an infinite measure. 
 For more details on this point, see Remark~\ref{last-rem}.
\end{remark}

\subsection{Proof of Theorem~A}\label{proofthma}
Let $a\in(0,\frac{1}{2m}]$.
By Proposition~\ref{exact}, $(f_a,{\rm Leb})$ is exact and so mixing of all orders. By Proposition~\ref{exact},
  $(g_{a},{\rm Leb})$ is mixing of all orders. By Proposition~\ref{dual}, $(g_{\frac{1}{m}-a},{\rm Leb})$ is mixing of all orders too, and so is $(f_{\frac{1}{m}-a},{\rm Leb})$. The proof of Theorem~A is complete. \qed

 \section{Exponential mixing for the heterochaos baker maps}
In this section we prove Theorem~B.
In Section~\ref{tower-sec} 
we introduce towers with exponential tail
associated with the heterochaos baker maps. In Section~\ref{proofthmb}
we apply the results in \cite{You98} to the towers and use large deviations
to complete the proof of Theorem~B.

\subsection{Towers}\label{tower-sec}
Put
$\Delta_0=\Delta^+_0\times[0,1]$.
We extend the stopping time 
 to a function on $\Delta_0$ in the obvious way, and still denote the extension by $R$: $R(x,y,z)=R(x,y)$ for $(x,y,z)\in\Delta_0$.
We define a {\it tower} $\Delta$ associated with $g=g_a$ by
\[\Delta=\{(p,\ell)\colon p\in\Delta_0,\ \ell=0,1,\ldots, R(p)-1\}.\]
For each $\ell\geq1$, the $\ell$-th floor is the set
\[\Delta_\ell=\{(p,\ell)\in\Delta\colon p\in\Delta_0\}.\]
We identify the ground floor $\Delta_0\times\{0\}$
with $\Delta_0$. Note that
$\Delta=\bigcup_{\ell=0}^\infty\Delta_\ell$.
  Define a {\it tower map} $G\colon\Delta\to\Delta$ by
\[ G(p,\ell)=\begin{cases}(p,\ell+1)&\ \text{ if }\ell+1<R(p),\\
(g^{R(p)}(p),0)&\ \text{ if }\ell+1=R(p).
\end{cases}\]
Collapsing the $z$-coordinate, we obtain 
the quotient tower $\Delta^+$ and the tower map
$F\colon \Delta^+\to\Delta^+$.
The maps $g$ and $G$, $f$ and $F$ are semiconjugated by the maps $\theta\colon(p,\ell)\in \Delta\mapsto g^\ell(p)\in[0,1]^3$ and 
$\theta^+\colon(p,\ell)\in \Delta^+\mapsto f^\ell(p)\in[0,1]^2$ respectively.
Let
${\rm pr}\colon\Delta\to\Delta^+$ denote the canonical projection. 
In summary, the following diagram commutes:
  \[
   \xymatrix{
    \Delta   \ar^{\begin{split}{\rm pr}\end{split}}[rd] \ar_{\begin{split}&\\
    &\theta\end{split}}[dd] \ar^{\begin{split}G\end{split}}[rr] & &
    \Delta \ar^{\begin{split}{&\\
    &\theta}\end{split}}[dd]|{\hole} \ar^{\begin{split}{\rm pr}\end{split}}[rd] &  \\
    & \Delta^+ \ar_{\begin{split}&\\
    &\theta^+\end{split}}[dd] \ar^{\begin{split}\!\!\!\!\!\!\!\!\!F\end{split}}[rr] & & \Delta^+ \ar^{\begin{split}&\\ &\theta^+\end{split}}[dd] \\
    [0,1]^3 \ar[rr]^{\begin{split}\ \ \ \ \ \  g\end{split}}|{\hole} \ar[rd] & &
    [0,1]^3 \ar[rd] & \\
    & [0,1]^2 \ar^{\begin{split}f\end{split}}[rr] & & 
   [0,1]^2
.}\]

We fix a sigma-algebra on $\Delta$
that is obtained by naturally transplanting the Borel sigma-algebra on $[0,1]^3$. 
Each floor $\Delta_\ell$, $\ell\geq0$ is identified with $\{R>\ell\}$, and so equipped with the restriction of the Lebesgue measure.
If $a\in(0,\frac{1}{2m})$, then in view of
  Proposition~\ref{tail-eq0} let
$\mu$ denote the probability measure on $\Delta$ given by
\[\mu(A)=\frac{1}{\int R{\rm d}{\rm Leb}} \sum_{\ell=0}^\infty|A\cap\Delta_\ell|\ \text{ for any measurable set } A\subset\Delta.\]
Since the Lebesgue measure on $[0,1]^3$ is 
$g$-invariant, $\mu$ is $G$-invariant.
The measure $\mu\circ\theta^{-1}$ is $f$-invariant, and absolutely continuous with respect to the Lebesgue measure, and hence
 $\mu\circ\theta^{-1}={\rm Leb}$.
 The measure $\mu^+=\mu\circ{\rm pr}^{-1}$
is $F$-invariant. 

\subsection{Proof of Theorem~B}\label{proofthmb}
Let $a\in(0,\frac{1}{2m})$. Let $\eta\in(0,1]$ and let $\varphi$, $\psi\in\mathscr H_\eta([0,1]^3)$.
For a function $\phi$ on $[0,1]^3$,
let $\tilde\phi$ denote its lift
to the tower $\Delta$ associated with $g=g_a$, i.e. $\tilde\phi=\phi\circ\theta$.
Note that
${\rm Cor}_n(g;\varphi,\psi;{\rm Leb})={\rm Cor}_n(G;\tilde\varphi,\tilde\psi;\mu)$. 

The rest of the proof of Theorem~B breaks into three steps,
 much in parallel to \cite[Section~4.1]{You98} with one important difference in Step~2.
In Step~1 we begin by approximating $\tilde\varphi\circ G^k$, $\tilde\psi\circ G^k$, $k\geq1$
by functions $\varphi_k$, $\psi_k$ on $\Delta$ which do not depend on the $z$-coordinate.
In Step~2 we provide error  bounds of these approximations.
Since \cite[p.608~Sublemma]{You98} does not hold as a result of the breakdown of
the condition (P4) (backward~contraction) in \cite{You98}, we estimate errors of these approximations using large deviations for the map $\tau_a$ in \eqref{Fa-1d}.
In Step~3
we view $\varphi_k$, $\psi_k$ as functions on $\Delta^+$, and show exponential decay of their correlations by applying 
\cite{You98}. We then unify all these estimates.
\medskip

\noindent{\it Step 1: Approximations.}
Put $\Delta_{0,i}=\Delta^+_{0.i}\times[0,1]$ for $i\geq1$.
Each floor $\Delta_\ell$, $\ell\geq0$ is  partitioned into $\{\Delta_{\ell,i}\}_{i\geq1\colon R_i>\ell}$ where $\Delta_{\ell,i}$ is a copy of $\Delta_{0,i}$.
Let $\mathscr D_0$ denote the partition of $\Delta$ into $\Delta_{\ell,i}$-components. It has the Markov property: For every $A\in\mathscr{D}_0$, $G(A)$ is the union of elements of $\mathscr{D}_0$. For $k\geq1$ we put
$\mathscr{D}_k=\bigvee_{j=0}^{k-1}G^{-j}\mathscr{D}_0$.
For $k\geq0$, let $\mathscr{D}_k^+$ denote the partition of $\Delta^+$ which is obtained as the canonical projection of $\mathscr{D}_k$.
For $\phi\in\mathscr H_\eta([0,1]^3)$ and $k\geq1$,
define $\phi_k\colon\Delta\to\mathbb R$
by 
\[\phi_k|_A=\inf\{\tilde\phi(w)\colon w\in G^k(A)\}\quad
\text{for every }   A\in\mathscr D_{2k}.\] 
Clearly $\phi_k$ is constant on each element of $\mathscr{D}_{2k}$, and  $|\phi_k|_\infty\leq|\phi|_\infty$.
Since $\phi_k$ does not depend on the $z$-coordinate,
we may view $\phi_k$ as a function on $\Delta^+$.
Let $n>k$.
By the identities
${\rm Cor}_n(G;\tilde\varphi,\tilde\psi;\mu)=
{\rm Cor}_{n-k}(G;\tilde\varphi,\tilde\psi\circ G^k;\mu)$ and
${\rm Cor}_n(G;\varphi_k,
\psi_k;\mu)={\rm Cor}_n(F;\varphi_k,
\psi_k;\mu^+)$,
we have
${\rm Cor}_n(G;\tilde\varphi,\tilde\psi;\mu)\leq I+I\!I+I\!I\!I$
where
\[\begin{split}I&=|{\rm Cor}_{n-k}(G;\tilde\varphi,\tilde\psi\circ G^k;\mu)-{\rm Cor}_{n-k}(G;\tilde\varphi,\psi_k;\mu)|,\\
I\!I&=|{\rm Cor}_{n-k}(G;\tilde\varphi,\psi_k;\mu)-{\rm Cor}_{n-k}(G;\varphi_k,\psi_k;
\mu)|,\\
I\!I\!I&={\rm Cor}_{n-k}(F;\varphi_k,
\psi_k;\mu^+).\end{split}\]

\noindent{\it Step 2: Estimates of $I$, $I\!I$.}
A direct calculation shows
$I\leq 2|\varphi|_\infty
\int|\tilde\psi\circ G^k-\psi_k|{\rm d}\mu,$ and a similar argument to the one 
 in \cite[p.608]{You98} shows
$I\!I\leq 2|\psi|_\infty
\int|\tilde\varphi\circ G^k-\varphi_k|{\rm d}\mu.$
Hence, upper bounds of $I$, $I\!I$ follow from the next lemma.
For $\phi\in\mathscr{H}_\eta([0,1]^3)$ let 
$|\phi|$ denote the $\eta$-H\"older norm of $\phi$, namely
\[|\phi|={\sup}_{\stackrel{p,q\in[0,1]^3 }{ p\neq q}}\frac{|\phi(p)-\phi(q)|}{|p-q|^\eta },\]
where $|p-q|$ denotes the Euclidean distance between $p$ and $q$. 
\begin{lemma}\label{LD-lem}
For any $\eta\in(0,1]$ and any $\phi\in \mathscr{H}_\eta([0,1]^3)$,
there exist constants $C=C(\phi)>0$ and $\xi=\xi(\eta)\in(0,1)$ such that 
for every $k\geq1$,
\[\begin{split}\int|\tilde\phi\circ G^k-\phi_k| {\rm d}\mu\leq C\xi^k.\end{split}\]
\end{lemma}
\begin{proof}
Put
\[\chi^u=-\log\max\left\{a,1-ma\right\},\
\chi^s=\log\max\left\{b,1-mb\right\},\
\chi^c=(1-2ma)\log m.\]
The minimal expansion rate by $g$ in the $x$-direction is $\exp(\chi^u)$
and the maximal contraction rate by $g$ in the $z$-direction is $\exp(\chi^s)$. 
Since $a\in(0,\frac{1}{2m})$ we have $\chi^c>0$.
Fix $\varepsilon\in(0,\chi^c)$,
and recall the definition of a block and the notation introduced in the beginning of Section~\ref{a-p}.
For each $A\in \mathscr{D}_k$, 
$\theta(A)$ is a block
satisfying
$|\theta(A)|_x\leq\exp(-\chi^uk)$ and $|\theta(A)|_z\leq\exp(\chi^sk).$
Define
\[\mathscr D_{k}'=\left\{A\in \mathscr D_{k}\colon
|\theta(A)|_y<\exp(-(\chi^c-\varepsilon) k)\right\}.\]
Put
$\xi_0=\exp(\max\{-\chi^u,-\chi^c+\varepsilon,\chi^s\})\in(0,1).$
For every $A\in\mathscr D_{2k}$ with $G^k(A)\in\mathscr{D}_k'$,
the Euclidean diameter of $\theta (G^k(A))$ does not exceed $3\xi_0^k$.
Since $\phi\in\mathscr{H}_\eta([0,1]^3)$, this bound implies
\begin{equation}\label{eqho1}\sup_A|\tilde\phi\circ G^k-\phi_k|\leq 3^\eta|\phi|\xi_0^{\eta k}.\end{equation}

For each $w\in\Delta^+$
let $\mathscr{D}_k^+(w)$ denote the element of $\mathscr{D}_k^+$ that contains $w$. 
Then $\mathscr{D}_k^+(w)$ is a rectangle satisfying
 $|\mathscr{D}_k^+(w)|_y\leq\exp(-S_k\phi^c(\theta^+(w)))$, which implies
\begin{equation}\label{eqho4}\begin{split}\sum_{A\in\mathscr{D}_{k}\setminus\mathscr{D}_k'}\mu(A)&\leq\mu^+\left\{w\in \Delta^+\colon S_k\phi^c(\theta^+(w)) \leq (\chi^c-\varepsilon)k\right\}\\
&=\left|\left\{p\in [0,1]^2\colon S_k\phi^c(p)\leq(\chi^c-\varepsilon)k\right\}\right|.\end{split}\end{equation}
Recall that $g$ is a skew product over $\tau$ in \eqref{Fa-1d} that preserves the Lebesgue measure on $[0,1]$, and 
$\int_{[0,1]}\phi^c(x,0){\rm d}x=\chi^c$.
From the large deviations applied to the sequence of independently identically distributed random variables
$x\in[0,1]\mapsto \phi^c(f^k(x,0))$ $(k=0,1,\ldots)$,
the sum in \eqref{eqho4} decays exponentially in $k$.
There exist $C>0$ and $\xi_1\in(0,1)$ such that
\begin{equation}\label{eqho3}
\sum_{A\in\mathscr{D}_{k}\setminus\mathscr{D}_k'}\mu(G^{-k}(A))\leq\sum_{A\in\mathscr{D}_{k}\setminus\mathscr{D}_k'}\mu(A)\leq C\xi_1^k.\end{equation}
The first inequality is due to the $G$-invariance of $\mu$. Let $D_k$ denote the union of elements of 
$\mathscr{D}_k'$.
For any $A\in\mathscr{D}_{2k}$, either $G^k(A)\subset D_k$ or $G^k(A)\subset \Delta\setminus D_k$ holds. 
We have
\begin{equation}\label{eqho39}
\sum_{A\in\mathscr{D}_{2k}\colon G^k(A)\subset \Delta\setminus D_k }\mu(A)\leq\sum_{A\in\mathscr{D}_{k}\setminus\mathscr{D}_k'}\mu(G^{-k}(A)).\end{equation}

From \eqref{eqho1}, \eqref{eqho4}, \eqref{eqho3} and
\eqref{eqho39} we obtain
\[\begin{split}\int|\tilde\phi\circ G^k-\phi_k| {\rm d}\mu=&\sum_{A\in\mathscr{D}_{2k}\colon G^k(A)\subset  D_k }\int_{A}|\tilde\phi\circ G^k-\phi_k| {\rm d}\mu\\
&+\sum_{A\in\mathscr{D}_{2k}\colon G^k(A)\subset \Delta\setminus D_k}\int_{A}|\tilde\phi\circ G^k-\phi_k| {\rm d}\mu\\
\leq& 3^\eta|\phi|\xi_0^{\eta k}\sum_{A\in\mathscr{D}_{2k}}\mu(A)+2|\phi|_\infty\sum_{A\in\mathscr{D}_{k}\setminus\mathscr{D}_k'}
\mu(A)\\
\leq& 3^{\eta}|\phi|\xi_0^{\eta k}+2C|\phi|_\infty\xi_1^k.\end{split}\]
Taking $C(\phi)=3^\eta|\phi|+2C|\phi|_\infty$ and $\xi=\max\{\xi_0^\eta,\xi_1\}$
yields the desired inequality in Lemma~\ref{LD-lem}.
\end{proof}


\noindent{\it Step3: Estimate of $I\!I\!I$ and an overall estimate.}
It is clear that
 the greatest common divisor of $\{R_i\colon i\geq1 \}$ is $1$.
Using the Perron-Frobenius operator
$\mathcal P\colon L^1(\mu^+)\to L^1(\mu^+)$ given by $F_*(\phi \mu^+)=\mathcal P(\phi)\mu^+,$ we write 
\[\begin{split}I\!I\!I&
={\rm Cor}_{n-2k}(F;\mathcal P^{2k}(\varphi_k),\psi_k;\mu^+).
\end{split}\]

We introduce a separation time $s\colon \Delta_0^+\times\Delta_0^+\to\mathbb Z$ by defining $s(p,p')$ to be the smallest integer $n\geq0$ such that $(f^R)^n(p)$, $(f^R)^n(p')$ lie in different elements of $\mathscr{P}|_{\Delta_0^+}$. For general points $w=(p,\ell)$,
 $w'=(p',\ell')\in\Delta^+$, define $s(w,w')=s(p,p')$ if $\ell=\ell'$ and $ s(w,w')=0$ otherwise.
 This defines a separation time $s\colon\Delta^+\times\Delta^+\to\mathbb Z$. For $\beta\in(0,1]$ define
\[C_\beta(\Delta^+)=\{\phi\colon\Delta^+\to\mathbb R\colon\exists C>0\  \forall w,w'\in\Delta^+,\ |\phi(w)-\phi(w')|\leq C\beta^{s(w,w')}\}.\]
For each $\phi\in C_\beta(\Delta^+)$
define 
\[\|\phi\|=|\phi|_\infty+{\rm esssup}_{\stackrel{w, w'\in\Delta^+}{w\neq w'}}\frac{|\phi(w)-\phi(w')|}{\beta^{s(w,w')}}.\]


\begin{lemma} 
For all $\beta\in(0,1)$ and all $k\geq1$, we have
$\mathcal P^{2k}(\varphi_k)\in C_\beta(\Delta^+)$
and
$\|\mathcal P^{2k}(\varphi_k)\|\leq 3|\varphi|_\infty$. \end{lemma}
\begin{proof}
For each $A\in\mathscr{D}_{k}^+$,  $F^{k}|_A$ has a measurable inverse,
and the pullback of the measure $\mu^+|_{F^k(A)}$ by
$F^k|_A$ is absolutely continuous with respect to 
$\mu^+|_A$.
The Radon-Nikod\'ym derivative 
${\rm d}(\mu^+|_{F^k(A)}\circ F^k|_A)/{\rm d}\mu^+|_A$
is constant on $A$, which we denote by 
$JF^k(A)$.
Note that
\[\mathcal P^{2k}(\varphi_k)(w)=\sum_{A\in\mathscr{D}_{2k}^+\colon w\in F^{2k}(A)}\frac{\varphi_k|_A}{JF^{2k}(A)}.\]
Moreover,
$JF^{2k}(A)=\mu^+(F^{2k}(A))/\mu^+(A)\geq1/\mu^+(A)$ for all $A\in\mathscr{D}_{2k}^+$, and so
$\sum_{A\in\mathscr{D}_{2k}^+} 1/JF^{2k}(A)\leq1.$
From this and $|\varphi_k|_\infty\leq|\varphi|_\infty$ we obtain
$|\mathcal P^{2k}(\varphi_k)|_\infty\leq|\varphi|_\infty$.

For all $w,w'\in\Delta^+$ with $w\neq w'$ we have
$\left|\mathcal P^{2k}(\varphi_k)(w)-
\mathcal P^{2k}(\varphi_k)(w')\right|\leq 2|\varphi|_\infty.$  If $s(w,w')\geq1$,
then
 for each $A\in\mathscr{D}_{2k}^+$ either
$w,w'\in F^{2k}(A)$ or $w$, $w'\notin F^{2k}(A)$, which implies  
$\left|\mathcal P^{2k}(\varphi_k)(w)-
\mathcal P^{2k}(\varphi_k)(w')\right|=0.$
Hence the conclusion of the lemma holds.
\end{proof}

By \cite[Section~4.2]{You98} (see also \cite[Theorem~1.5]{MD01}), there exist $C>0$ and $\lambda_0\in(0,1)$ such that
\[I\!I\!I\leq 
C\lambda_0^{n-2k}.
\]
Combining this with the estimates of $I$, $I\!I$ in Step~2, and then substituting
$k=\lfloor n/2\rfloor/2$,
we obtain constants $C=C(\varphi,\psi)>0$ and $\lambda\in(0,1)$ such that \[{\rm Cor}_{n}(g;\varphi,\psi;{\rm Leb})={\rm Cor}_{n}(G;\tilde\varphi,\tilde\psi;\mu)\leq C\lambda^n\text{ for all } n\geq1,\] 
namely, the exponential mixing for $(g_a,{\rm Leb})$.
From this and Proposition~\ref{dual}
  we obtain the exponential mixing for $(g_{\frac{1}{m}-a},{\rm Leb})$.
Taking a pair of H\"older continuous functions on $[0,1]^3$  which are constant on sets 
$\{(x,y)\}\times[0,1]$, $(x,y)\in[0,1]^2$,
we obtain the exponential mixing  for $(f_a,{\rm Leb})$ and 
$(f_{\frac{1}{m}-a},{\rm Leb})$.
The proof of Theorem~B is complete.
\qed

\begin{remark}\label{last-rem}
In the case $a=\frac{1}{2m}$ which is excluded in Theorem~B,
it is plausible that correlations for H\"older continuous functions
decay only subexponentially. It would be nice to prove this.
A close inspection into the counting argument in Section~\ref{count-sec} and the stopping time estimate in Section~\ref{ind-exp-sec} reveals 
$|\{R=n\}|\asymp n^{-3/2}$, and as a result 
$\sum_{n}|\{R>n\}|=\infty$. It follows that  the lift of the Lebesgue measure to the tower becomes an infinite measure, and the general result in \cite{You99} to draw subexponential decay of correlations is not applicable.
\end{remark}

\section{Exponential mixing for the Dyck system}\label{exp-mix}

In this last section we prove Theorem~C. 
In Section~\ref{Dyck-mme}, we identify the two ergodic measures of maximal entropy for the Dyck system. In Section~\ref{cor-Dyck-sec} we show the invariance of correlations (modulo involution) under the replacement of the two measures. 
Section~\ref{transfer-sec} provides two preliminary lemmas needed to transfer invariant measures on the two different spaces.
In Section~\ref{ent-Leb} we clarify a connection between one of the two ergodic measures of maximal entropy and the Lebesgue measure on $[0,1]^3$. In Section~\ref{pfthmc} we complete the proof of Theorem~C.

\subsection{Measures of maximal entropy}\label{Dyck-mme}
Let $T$ be a Borel map acting on a topological space.  
For each $T$-invariant Borel probability measure $\mu$, let $h(T,\mu)$ denote the measure-theoretic entropy of $\mu$ with respect to $T$.  If $\sup\{h(T,\mu)\colon\text{$\mu$ is $T$-invariant}\}$ is finite, a measure which attains this supremum is called {\it a measure of maximal entropy}.

Following \cite[Section~4]{Kri74}, for each $i\in\mathbb Z$ we define $H_i\colon \Sigma_D\to\mathbb Z$ by
      \[H_i(\omega)=\begin{cases}\sum_{j=0}^{i-1} \sum_{l=1}^{m}(\delta_{{\alpha_l},\omega_j}-\delta_{\beta_l,\omega_j})&\text{ for }  i\geq1,\\\sum_{j=i}^{-1} \sum_{l=1}^{m}(\delta_{{\beta_l},\omega_j}-\delta_{\alpha_l,\omega_j})&\text{ for } i\leq -1,\\
      0&\text{ for }i=0,\end{cases}\]
      where the delta denotes Kronecker's delta.
       These functions are used to indicate whether a bracket in a prescribed position in a sequence in $\Sigma_D$ is closed or not.
  Consider three shift invariant Borel sets
    \begin{equation}\label{three-sets}\begin{split}
    A_0&=\bigcap_{i=-\infty}^\infty\left(\left(\bigcup_{l=1}^\infty\{ H_{i+l}=H_i\}\right)\cap\left(\bigcup_{l=1}^\infty\{ H_{i-l}=H_i\}\right)\right),\\
A_\alpha&=\left\{\omega\in\Sigma_D\colon
\lim_{i\to\infty}H_i(\omega)
=\infty\ \text{ and } \ \lim_{i\to-\infty}H_i(\omega)=-\infty\right\},\\
A_\beta&=\left\{\omega\in\Sigma_D\colon
\lim_{i\to\infty}H_i(\omega)=-\infty\ \text{ and } \
\lim_{i\to-\infty}H_i(\omega)=\infty\right\}.\end{split}\end{equation}
Any shift invariant ergodic measure on $\Sigma_D$ gives measure $1$ to one of these three sets (see \cite[pp.102--103]{Kri74}).
\begin{thm}[\cite{Kri74}, Section~4]\label{idem-mme}
There exist two ergodic measures $\nu_\alpha$, $\nu_\beta$ of entropy $\log(m+1)$ which are Bernoulli and satisfy 
$\nu_\alpha(A_{\alpha})=1$ and $\nu_\beta(A_{\beta})=1$.
These two measures are precisely all the ergodic measures of maximal entropy for the two-sided Dyck shift.
\end{thm}

Let $\nu_\alpha^+$ (resp. $\nu_\beta^+$) denote the shift invariant measure on $\Sigma_D^+$
which is the push-forward of $\nu_\alpha$ (resp. $\nu_\beta$) under the canonical projection $\Sigma_D\to\Sigma_D^+$.
These two measures are precisely all the ergodic measures of maximal entropy for the one-sided Dyck shift.

We recall the result in \cite{TY} on measures of maximal entropy for the heterochaos baker maps.
Let $c_1=\frac{1}{m(m+1)}$ and $c_2=\frac{1}{m+1}$.
 \begin{thm}[\cite{TY}, Theorem~1.2]\label{thm1.2}
For any $a,b\in(0,\frac{1}{m})$,
there exist two $f_{a,b}$-invariant
 ergodic Borel probability measures $\mu_\alpha$, $\mu_\beta$ 
 of entropy $\log (m+1)$ which are Bernoulli, 
 give positive weight to any non-empty open subset of $[0,1]^3$, and satisfy
\[\mu_\alpha\left(\bigcup_{i=1}^m\Omega_{\alpha_i}\right)=
\mu_\beta\left(\bigcup_{i=1}^{m}
\Omega_{\beta_i}\right)=\frac{m}{m+1}.\]
Moreover, if $a\in[c_1,c_2]$ or $b\in[c_1,c_2]$
then $\mu_\alpha$, $\mu_\beta$ are measures of maximal entropy for $f_{a,b}$. 
If $a\in(c_1,c_2)$ or $b\in(c_1,c_2)$ then
there is no  ergodic measure of maximal entropy for $f_{a,b}$ other than $\mu_\alpha$, $\mu_\beta$.
\end{thm}

\subsection{Invariance of correlations}\label{cor-Dyck-sec} 
We define an involution $\rho\colon D\to D$ by $\rho (\alpha_i)=\beta_i$ and $\rho(\beta_i)=\alpha_i$
for $i\in\{1,\ldots,m\}$, and $\iota_D\colon \Sigma_D\to D^\mathbb Z$
by $\iota_D((\omega_n)_{n\in\mathbb Z})=(\rho(\omega_{-n}))_{n\in\mathbb Z}$. Clearly $\iota_D$ is injective, and 
$\iota_D(\Sigma_D)=\Sigma_D$
as in Lemma~\ref{iotaD-lem} below.
\begin{prop}
\label{dual-Dyck}
For all $\varphi,\psi\in L^2(\nu_\beta)$ and all $n\geq1$ 
we have
\[{\rm Cor}_n(\sigma;\varphi,\psi;\nu_\beta)={\rm Cor}_{n}(\sigma;\psi\circ\iota_D,\varphi\circ\iota_D;\nu_\alpha).\]
\end{prop}

\begin{proof}By Lemma~\ref{iotaD-lem} below,
for any shift invariant measure $\nu$ on $\Sigma_D$
we have
   $h(\sigma,\nu)=h(\sigma,\nu\circ\iota_D^{-1})$ and $\nu(\bigcup_{i=1}^m[\alpha_i])
    =\nu\circ\iota_D^{-1}(\bigcup_{i=1}^m[\beta_i])$ where
    $[\gamma]=\{(\omega_n)_{n\in\mathbb Z}\colon\omega_0=\gamma\}$ for $\gamma\in D$.
This implies $\nu_{\beta}=\nu_{\alpha}\circ\iota_D^{-1}$. 
We have
\[\begin{split}\int\varphi(\psi\circ\sigma^n){\rm d}\nu_\beta&=\int\varphi\circ\iota_D(\psi\circ\sigma^n\circ\iota_D){\rm d}\nu_\alpha\\
&=\int\varphi\circ\iota_D(\psi\circ\iota_D\circ\sigma^{-n}){\rm d}\nu_\alpha\\
&=\int\varphi\circ\iota_D\circ\sigma^n(\psi\circ\iota_D){\rm d}\nu_\alpha,\end{split}\]
  and $\int\varphi{\rm d}\nu_\beta=\int\varphi\circ\iota_D{\rm d}\nu_\alpha$ and
$\int\psi{\rm d}\nu_\beta=\int\psi\circ\iota_D{\rm d}\nu_\alpha$. Hence the desired equality holds.
\end{proof}

\begin{lemma}\label{iotaD-lem} 
We have $\iota_D(\Sigma_D)=\Sigma_D$, and  $\iota_D\circ\sigma=\sigma^{-1}\circ\iota_D$. 
\end{lemma}

\begin{proof}
The second assertion of the lemma is a consequence of the first one, which was proved in \cite[Section~3.5]{TY}. We include the proof here for the reader's convenience.
For $n\geq1$ and $\lambda=\gamma_1\cdots \gamma_n\in L(\Sigma_D)$ we set
$\rho^{\ast}(\lambda)=\rho(\gamma_n)\cdots \rho(\gamma_1)$. It is enough to show that 
$\rho^{\ast}(\lambda)\in L(\Sigma_D)$, namely
${\rm red}(\rho^{\ast}(\lambda))\neq0$.
 By the relations \eqref{d-rel}, either (i) ${\rm red}(\lambda)=1$, or (ii) 
${\rm red}(\lambda)=\xi\eta$
for some $\xi\in L(\{\beta_1,\ldots,\beta_m\}^{\mathbb{Z}})$ and
$\eta\in L(\{\alpha_1,\ldots,\alpha_m\}^{\mathbb{Z}})$.
 In case (i), clearly we have ${\rm red}(\rho^{\ast}(\lambda ))=1$.
In case (ii), we have 
 $\rho^{\ast}(\eta)\in
L(\{\beta_1,\ldots,\beta_m\}^{\mathbb{Z}})$ and
 $\rho^{\ast}(\xi)\in
L(\{\alpha_1,\ldots,\alpha_m\}^{\mathbb{Z}})$, and so
${\rm red}(\rho^{\ast}(\lambda))=
\rho^{\ast}(\eta)\rho^{\ast}(\xi)\neq0$.
 \end{proof}

\subsection{Transferring invariant measures}\label{transfer-sec}
In order to transfer invariant measures on the two different spaces
under the coding map, we need two lemmas.
We extend the central Jacobian to a function on $[0,1]^3$ in the obvious way, and still denote the extension by $\phi^c$: $\phi^c(x,y,z)=\phi^c(x,y)$ for $(x,y,z)\in[0,1]^3$.
\begin{lemma}\label{class-lem}
Let $a,b\in(0,\frac{1}{m})$ and
    let $\mu$ be an $f_{a,b}$-invariant ergodic Borel probability measure
    satisfying $\mu(\Lambda_{a,b})=1$.
    \begin{itemize}
    \item[(a)] If $\int\phi^c{\rm d}\mu=0$ then $\mu\circ\pi^{-1}(A_0)=1$.
    \item[(b)] If
    $\int\phi^c{\rm d}\mu<0$ then $\mu\circ\pi^{-1}(A_\alpha)=1$.
    \item[(c)] If
    $\int\phi^c{\rm d}\mu>0$ then $\mu\circ\pi^{-1}(A_\beta)=1$.\end{itemize}
\end{lemma}
\begin{proof}
Since shift invariant ergodic measures on $\Sigma_D$ give measure $1$ to one of the three sets in \eqref{three-sets} (see \cite[pp.102--103]{Kri74}),
the assertions are consequences of the definition of $H_i$, $\phi^c$ and Birkhoff's ergodic theorem.
\end{proof}

Let
\[A_{\alpha,\beta}=\left\{\omega\in\Sigma_D\colon \liminf_{i\to\infty}H_i(\omega)=-\infty\ \text{ or } \ \liminf_{i\to-\infty}H_i(\omega)=-\infty\right\}.\]
Note that $A_{\alpha,\beta}$ is shift invariant and contains $A_\alpha\cup A_\beta$.
\begin{lemma}[\cite{TY}, Lemma~3.1]\label{m-bij}
For all $a,b\in(0,\frac{1}{m})$,
the restriction of  $\pi$ to $\pi^{-1}(A_{\alpha,\beta})$ is a homeomorphism onto its image.
\end{lemma}

 \subsection{Connection to the Lebesgue measure}\label{ent-Leb}
 Since $c_1+c_2=\frac{1}{m}$, $f_{c_1,c_2}$ and $f_{c_2,c_1}$ preserve the Lebesgue measure on $[0,1]^3$.

\begin{prop}\label{correspond}
We have
$h(g_{c_1},{\rm Leb})=h(g_{c_2},{\rm Leb})=\log (m+1)$, 
${\rm Leb}\circ\pi_{c_2,c_1}^{-1}=\nu_\alpha$ and 
 ${\rm Leb}\circ\pi_{c_1,c_2}^{-1}=\nu_\beta.$
\end{prop}

\begin{proof} 
Recall that the restriction of $g_{c_1}$ to $\Lambda_{c_1,c_2}$ is invertible,
 uniformly expanding in the $x$-direction and uniformly contracting in the $z$-direction. Using Proposition~\ref{R-fin} to deal with the dynamics in the $y$-direction,
it is easy to see that 
 \[\bigvee_{n=-\infty}^{\infty}g^{-n}\{\Omega_\gamma\cap\Lambda_{c_1,c_2}\}_{\gamma\in D}\stackrel{\circ}{=}\mathscr B(\Lambda_{c_1,c_2}).\]
 Moreover, $(g_{c_1},{\rm Leb})$ is ergodic by Theorem~A.
  Therefore, Shannon-McMillan-Breimann's theorem yields
\[h(g_{c_1},{\rm Leb})=-\lim_{n\to\infty}\frac{1}{n}\log \left|\bigcap_{k=0}^{n-1} g_{c_1}^{-k}\left(\Omega_{\omega_k }\right)\right|\]
 for Lebesgue a.e. $p\in\Lambda_{c_1,c_2 }$
where $\pi_{c_1,c_2 }(p)=(\omega_n)_{n=-\infty}^\infty\in\Sigma_{D}$. Clearly
 we have
 \[\begin{split}-\lim_{n\to\infty}\frac{1}{n}\log\left|\bigcap_{k=0}^{n-1} g_{c_1}^{-k}\left(\Omega_{\omega_k}\right)\right|\geq &mc_1\log\frac{1}{c_1}+(1-mc_1)\log\frac{1}{1-mc_1}\\&+(1-2mc_1)\log m=\log(m+1),\end{split}\]
and hence
  $h(g_{c_1},{\rm Leb})\geq\log (m+1)$. The reverse inequality is a consequence of Theorem~\ref{thm1.2}. 
  We have verified $h(g_{c_1},{\rm Leb})=\log (m+1)$.
By \cite[Lemma~3.9]{TY} and
${\rm Leb}\circ\iota^{-1}={\rm Leb}$ we have
 $h(g_{c_2},{\rm Leb})=h(g_{c_1},{\rm Leb})$.

By Theorem~A and Theorem~\ref{TY-thm},
the measure
${\rm Leb}\circ\pi_{c_2,c_1}^{-1}$ on $\Sigma_D$ is shift invariant and ergodic. By $\int\phi^c{\rm d}{\rm Leb}=-\frac{m-1}{m+1}\log2<0$ and Lemma~\ref{class-lem}, it
gives measure $1$ to $A_{\alpha}$, and by Lemma~\ref{m-bij} has entropy equal to
    $h(g_{c_2},{\rm Leb})$, which equals $\log(m+1)$ as already proved.
Hence we obtain
${\rm Leb}\circ\pi_{c_2,c_1}^{-1}=\nu_\alpha$.
    A proof of the last equality in the proposition is analogous.
\end{proof}

\subsection{Proof of Theorem~C}\label{pfthmc}
First 
we show exponential mixing for $(\sigma,\nu_\alpha)$, $(\sigma,\nu_\beta)$.
To this end, by virtue of Proposition~\ref{dual-Dyck} it suffices to show that for any pair $\varphi$, $\psi$  
of H\"older continuous functions on $\Sigma_D$, their correlation
${\rm Cor}_n(\sigma;\varphi,\psi;\nu_\beta)$ 
decays exponentially in $n$. This does not immediately follow from Theorem~B.
Indeed, Proposition~\ref{correspond}
implies
${\rm Cor}_n(\sigma;\varphi,\psi;\nu_\beta)={\rm Cor}_{n}(g_{c_1};\varphi\circ\pi_{c_1,c_2},\psi\circ\pi_{c_1,c_2};{\rm Leb})$,
but  $\varphi\circ\pi_{c_1,c_2}$ and
 $\psi\circ\pi_{c_1,c_2}$ may not be H\"older continuous. 
 Our strategy is to transfer to the Dyck system
the towers associated with $f_{c_1}$, $g_{c_1}$ constructed in Section~\ref{tower-sec}, and
mimic the proof of Theorem~B.

We restrict ourselves to the shift invariant set $A_\beta$ in \eqref{three-sets}. Bear in mind that $\nu_\beta(A_\beta)=1$.
By Lemma~\ref{m-bij},
the restriction of $\pi_{c_1,c_2}$ to $\pi_{c_1,c_2}^{-1}(A_{\beta})$ is a homeomorphism onto its image.
  We put
\[\hat\Delta_0=\pi_{c_1,c_2}(\Delta_0)\cap A_{\beta}\ \text{ and } \
\hat\Delta_{0,i}=\pi_{c_1,c_2}(\Delta_{0,i})\cap A_{\beta}.\]
Define
$\hat R\colon \pi_{c_1,c_2}([0,1]^3)\cap A_{\beta}\to\mathbb Z_+\cup\{\infty\}$ by $\hat R(\omega)=R(\pi_{c_1,c_2}^{-1}(\omega))$.
We define a tower 
\[\hat\Delta=\{(\omega,\ell)\colon \omega\in\hat\Delta_0,\ \ell=0,1,\ldots, \hat R(\omega)-1\},\]
and
  define a tower map $\hat G\colon\hat\Delta\to\hat\Delta$ by
\[\hat G(\omega,\ell)=\begin{cases}(\omega,\ell+1)&\ \text{ if }\ell+1<\hat R(\omega),\\
(\sigma^{\hat R(\omega)}(\omega),0)&\ \text{ if }\ell+1=\hat R(\omega).
\end{cases}\]
Collapsing the negative coordinate, we obtain 
the quotient tower $\hat{\Delta}^+$ and the tower map
$\hat F\colon \hat{\Delta}^+\to\hat{\Delta}^+$.
We fix a sigma-algebra on $\hat\Delta$
that is obtained by naturally transplanting the Borel sigma-algebra on $\Sigma_D$. 

By Proposition~\ref{R-fin}, $\hat R$ is finite $\nu_\beta$-a.e. and 
$\nu_\beta\{\hat R=n\}=|\{R=n\}|$.
Each floor \[\hat\Delta_\ell=\{(p,\ell)\in\hat\Delta\colon p\in\hat\Delta_0\},\ \ell\geq0\]
can be identified with $\{\hat R>\ell\}$, and so equipped with the restriction of the measure $\nu_\beta$. 
Let
$\hat\mu$ denote the probability measure on $\hat\Delta$ given by
\[\hat\mu(A)=\frac{1}{\int \hat R{\rm d}\nu_\beta }\sum_{\ell=0}^\infty\nu_\beta(A\cap\hat\Delta_\ell)\ \text{ for any measurable set }A\subset\hat\Delta.\]

We repeat Steps~1 to 3 in the proof of Theorem~B in Section~\ref{proofthma}.
Each floor $\hat\Delta_\ell$ is partitioned into $\{\hat\Delta_{\ell,i}\}_{i\geq1\colon \hat R_i>\ell}$ where $\hat\Delta_{\ell,i}$ is a copy of $\hat\Delta_{0,i}$.
Let $\hat{\mathscr D}_0$ denote the partition of $\hat\Delta$ into $\hat\Delta_{\ell,i}$-components. For $k\geq1$, let
$\hat{\mathscr{D}}_k=\bigvee_{j=0}^{k-1}{\hat G}^{-j}\hat{\mathscr{D}}_0$.
For $\phi\in\mathscr H_\eta(\Sigma_D)$ and $k\geq1$,
define $\phi_k\colon\hat\Delta\to\mathbb R$
by 
$\phi_k|_A=\inf\{\tilde\phi(w)\colon w\in \hat G^k(A)\}$ for $A\in\hat{\mathscr D}_{2k}.$
Define $\hat\theta\colon(\omega,\ell)\in \hat\Delta\mapsto \sigma^\ell(\omega)\in\Sigma_D$.
For every $A\in\hat{\mathscr D}_{2k}$, the diameter of the set
$\hat\theta ({\hat G}^k(A))$ with respect to the Hamming metric $d$ does not exceed $e^{-k}$. As a counterpart of \eqref{eqho1} we have
$\sup_A|\tilde\phi\circ \hat G^k-\phi_k|\leq |\phi|e^{-\eta k}$, where
$|\phi|$ denotes the $\eta$-H\"older norm of $\phi$
with respect to $d$.
Therefore, 
as a counterpart of Lemma~\ref{LD-lem} we obtain
$\int|\tilde\phi\circ {\hat G}^k-\phi_k| {\rm d}\hat\mu\leq |\phi|e^{-\eta k}$.
To finish, the rest of the argument is completely analogous to Step~3 in the proof of Theorem~B.

We have verified exponential mixing for both $(\sigma,\nu_\alpha)$ and $(\sigma,\nu_\beta)$. Taking a pair of H\"older continuous functions on $\Sigma_D$ which depend only on positive coordinates, we obtain
exponential mixing for both $(\sigma_+,\nu_\alpha^+)$ and $(\sigma_+,\nu_\beta^+)$.
The proof of Theorem~C is complete.
\qed


 \subsection*{Acknowledgments}
The author thanks Yoshitaka Saiki, Toshi Sugiyama, Masato Tsujii, Kenichiro Yamamoto, James A. Yorke for fruitful discussions. 
This research was supported by the JSPS KAKENHI 19K21835, 20H01811.

\end{document}